\documentclass{article}

\usepackage{cite}
\usepackage[pdftex]{graphicx}
\usepackage{amsthm}
\usepackage{amsmath}
\usepackage{amssymb}
\usepackage{algorithmic}
\usepackage{algorithm}
\usepackage{array}
\usepackage{bm}
\usepackage{url}

\DeclareMathOperator*{\grad}{grad}

\newcommand{\stiefel}{\mathrm{St}}
\newcommand{\real}{\mathbb{R}}
\newcommand{\nat}{\mathbb{N}}

\newcommand{\norm}[2]{\left\lVert{#1}\right\rVert_{#2}}
\newcommand{\ip}[3]{\left\langle{#1,#2}\right\rangle_{#3}}

\newcommand{\expect}[2]{\mathbb{E}_{#1}\left[{#2}\right]}
\newcommand{\relmiddle}[1]{\mathrel{}\middle#1\mathrel{}}

\newtheorem{definition}{Definition}[section]
\newtheorem{theorem}{Theorem}[section]
\newtheorem{lemma}{Lemma}[section]
\newtheorem{corollary}{Corollary}[section]

\newtheorem{assumption}{Assumption}[section]
\newtheorem{problem}{Problem}[section]

\begin{document}
\title{Riemannian Adaptive Optimization Algorithm and \\
 Its Application to Natural Language Processing}

\author{Hiroyuki Sakai, Hideaki Iiduka\thanks{This work was supported by JSPS KAKENHI Grant Number JP18K11184.}}

\maketitle

\abstract
This paper proposes a Riemannian adaptive optimization algorithm to optimize the parameters of deep neural networks.
The algorithm is an extension of both AMSGrad in Euclidean space and RAMSGrad on a Riemannian manifold.
The algorithm helps to resolve two issues affecting RAMSGrad. The first is that it can solve the Riemannian stochastic optimization problem directly,
in contrast to RAMSGrad which only achieves a low regret.
The other is that it can use constant learning rates, which makes it implementable in practice.
Additionally, we apply the proposed algorithm to Poincar{\'e} embeddings that embed the transitive closure of the WordNet nouns into the Poincar{\'e} ball model of hyperbolic space.
Numerical experiments show that regardless of the initial value of the learning rate, our algorithm stably converges to the optimal solution and converges faster than the existing algorithms.

\section{Introduction} 
Riemannian optimization has attracted a great deal of attention \cite{iosifidis2015graph, mao2016novel, shen2018deep} in light of developments in machine learning and deep learning. This paper focuses on Riemannian adaptive optimization algorithms for solving an optimization problem on a Riemannian manifold. In the field of machine learning, there is an important example of the Riemannian optimization problem. Nickel and Kiela \cite{nickel2017poincare} proposed Poincar{\'e} embeddings, which embed hierarchical representations of symbolic data (e.g., text, graph data) into the Poincar{\'e} ball model of hyperbolic space. In fact, experiments on transitive closure of the WordNet noun hierarchy showed that embeddings into a 5-dimensional Poincar{\'e} ball are better than embeddings into a 200-dimensional Euclidean space. Since the Poincar{\'e} ball has a Riemannian manifold structure, the problem of finding Poincar{\'e} embeddings should be considered to be a Riemannian optimization problem.

Bonnabel \cite{bonnabel2013stochastic} proposed Riemannian stochastic gradient descent (RSGD), the most basic Riemannian stochastic optimization algorithm. RSGD is a simple algorithm, but its slow convergence is problematic. In \cite{sato2019riemannian}, Sato, Kasai, and Mishra proposed the Riemannian stochastic variance reduced gradient (RSVRG) algorithm and gave a convergence analysis under some natural assumptions. RSVRG converges to an optimal solution faster than RSGD; however, RSVRG needs to calculate the full gradient every few iterations. In Euclidean space, adaptive optimization algorithms, such as AdaGrad \cite{duchi2011adaptive}, Adam \cite[Algorithm 1]{kingma2015adam}, Adadelta \cite{zeiler2012adadelta}, and AMSGrad \cite[Algorithm 2]{reddi2018convergence}, \cite[Algorithm 1]{iiduka2020appropriate}, are widely used for training deep neural networks. However, these adaptive algorithms cannot be naturally extended to general Riemannian manifolds, due to the absence of a canonical coordinate system. Therefore, special measures are required to extend the adaptive algorithms to Riemannian manifolds. For instance, Kasai, Jawanpuria, and Mishra \cite{kasai2019riemannian} proposed adaptive stochastic gradient algorithms on Riemannian matrix manifolds by adapting the row, and column subspaces of gradients.

In the particular case of a product of Riemannian manifolds, B{\'e}cigneul and Ganea \cite{becigneul2018riemannian} proposed Riemannian AMSGrad (RAMSGrad) by regarding each component of the product Riemannian manifold as a coordinate component in Euclidean space. However, their convergence analysis had two points requiring improvement. First, they only performed a regret minimization (Theorem \ref{thm:old}) and did not solve the Riemannian optimization problem. Second, they did a convergence analysis with only a diminishing learning rate; i.e., they did not perform a convergence analysis with a constant learning rate. Since diminishing learning rates are approximately zero after a large number of iterations, algorithms that use them are not implementable in practice. In contrast, a constant learning rate does not cause this problem.

The motivation of this work is to identify whether or not RAMSGrad can be modified in such a way that it can be applied to Riemannian optimization from the viewpoints of both theory and practice. The theoretical motivation is to show that the modified RAMSGrad can solve directly the Riemannian optimization problem, while the practical motivation is to show that it can be applied to important problems in natural language processing and principal component analysis. 

Motivated by the above discussion, we propose modified RAMSGrad (Algorithm \ref{RAMSGrad}), which is an extension of RAMSGrad, to solve the Riemannian optimization problem (Problem \ref{pbl:main}). In addition, we give a convergence analysis (Theorem \ref{thm:main}) valid for both a constant learning rate (Corollary \ref{cor:CLR}) and diminishing learning rate (Corollary \ref{cor:DLR}). The analysis leads to the finding that the proposed algorithm can solve directly the Riemannian optimization problem. In particular, we emphasize that the proposed algorithm can use a constant learning rate to solve the problem (Corollary \ref{cor:CLR}), in contrast to the previous result \cite{becigneul2018riemannian} in which RAMSGrad with a diminishing learning rate only performed a regret minimization (see Subsection \ref{subsec:3.2} for comparisons of the proposed algorithm with RAMSGrad). In numerical experiments, we apply the proposed algorithm to Poincar{\'e} embeddings and compare it with RSGD, Riemannian AdaGrad (RAdaGrad) \cite[Section 3.2]{becigneul2018riemannian}, and Riemannian Adam (RAdam) \cite[Section 4]{becigneul2018riemannian} (Section \ref{sec:NE}). We show that it converges to the optimal solution faster than the existing algorithms and that it minimizes the objective function regardless of the initial learning rate. In particular, we show that the proposed algorithm with a constant learning rate is a good way of embedding the WordNet mammals subtree into a Poincar{\'e} ball. We also applied the algorithm to principal component analysis and found that the choice between using a constant or a diminishing learning rate depends on the dataset. These numerical comparisons lead to the finding that the proposed algorithm is good for solving important problems in natural language processing and principal component analysis.

This paper is organized as follows. Section \ref{sec:MathP} gives the mathematical preliminaries and states the main problem. Section \ref{sec:ALROA} describes the modified RAMSGrad and gives its convergence analysis. Section \ref{sec:NE} numerically compares the performance of the proposed learning algorithms with the existing algorithms. Section \ref{sec:CaFW} concludes the paper with a brief summary.

\section{Mathematical Preliminaries}
\label{sec:MathP}
\subsection{Definitions, assumptions, and main problem}
\label{subsec:2.1}
Let $M$ be a Riemannian manifold. An exponential map at $x \in M$, written as $\exp_{x} \colon T_xM \rightarrow M$, is a mapping from the tangent space $T_xM$ to $M$ with the requirement that a vector $\xi \in T_xM$ is mapped to the point $y := \exp_x(\xi) \in M$ such that there exists a geodesic $\gamma \colon [0,1] \rightarrow M$, which satisfies $\gamma(0)=x$, $\gamma(1)=y$, and $\dot{\gamma}(0)=\xi$, where $\dot{\gamma}$ is the derivative of $\gamma$ (see \cite{absil2008, zhang2016first}). Moreover, $\log_x \colon M \rightarrow T_xM$ denotes a logarithmic map at a point $x \in M$, which is defined as the inverse mapping of the exponential map at $x \in M$. For all $x,y \in M$, the existence of $\log_x (y)$ is guaranteed \cite[Chapter V, Theorem 4.1]{sakai1996riemannian} \cite[Proposition 2.1]{chong2010iterative}.

Next, we give the definitions of a geodesically convex set and function (see \cite[Section 2]{zhang2016first}) that generalize the concepts of a convex set and function in Euclidean space.
\begin{definition}
[Geodesically convex set] Let $X$ be a subset of a Riemannian manifold $M$. $X$ is said to be geodesically convex if, for any two points in $X$, there is a unique minimizing geodesic within $X$ which joins those two points.
\end{definition}
\begin{definition}
[Geodesically convex function] A smooth function $f \colon M \rightarrow \real$ is said to be geodesically convex if, for any $x,y \in M$, it holds that
\begin{align*}
f(y) \geq f(x) + \ip{\grad{f(x)}}{\log_x(y)}{x},
\end{align*}
where $\ip{\cdot}{\cdot}{x}$ is the Riemannian metric on $M$, and $\grad f(x)$ is the Riemannian gradient of $f$ at a point $x \in M$ (see \cite{absil2008}).
\end{definition}
For $i \in \{1,2, \cdots ,N\}$, let $M_i$ be a Riemannian manifold and $M$ be the Cartesian product of $n$ Riemannian manifolds $M_i$ (i.e., $M:=M_1 \times \cdots \times M_N$). $x^i \in M_i$ denotes a corresponding component of $x \in M$, and $\ip{\cdot}{\cdot}{x^i}$ denotes a Riemannian metric at a point $x^i \in M_i$. Furthermore, $\norm{\cdot}{x^i}$ represents the norm determined from the Riemann metric $\ip{\cdot}{\cdot}{x^i}$. For a geodesically convex set $X_i \subset M_i$, we define the projection operator as $\Pi_{X_i}:M_i \rightarrow X_i$; i.e., $\Pi_{X_i}(x^i)$ is the unique point $y^i \in X_i$ minimizing $d^i(x^i,\cdot)$, where $d^i(\cdot ,\cdot):M_i \times M_i \rightarrow \real$ denotes the distance function of $M_i$. The tangent space at a point $x = (x^1,x^2, \cdots ,x^N) \in M$ is given by $T_xM = T_{x^1}M_1 \oplus \cdots \oplus T_{x^N}M_N$, by considering $T_{x^i}M_i$ to be a subspace of $T_xM$, where $\oplus$ is the direct sum of vector spaces. Then, for a point $x = (x^1,x^2, \cdots ,x^N) \in M$ and a tangent vector $\xi \in T_xM$, we write $\xi =(\xi^i)=(\xi^1,\xi^2, \cdots ,\xi^N)$, where $i \in \{1,2 \cdots ,N\}$, and $\xi^i \in T_{x^i}M_i$. Finally, for $x^i,y^i \in M_i$, $\varphi_{x^i \rightarrow y^i}^i$ denotes an isometry from $T_{x^i}M_i$ to $T_{y^i}M_i$ (e.g., $\varphi^i_{x^i \rightarrow y^i}$ stands for parallel transport from $T_{x^i}M_i$ to $T_{y^i}M_i$).

$\expect{}{X}$ denotes the expectation of a random variable $X$, and $t_{[n]}$ denotes the history of the process up to time $n$ (i.e., $t_{[n]}:=(t_1,t_2, \cdots ,t_n)$). $\mathbb{E}[X | t_{[n]} ]$ denotes the conditional expectation of $X$ given $t_{[n]}$.

\begin{assumption}~\label{asm:main}
For $i \in \{1,2, \cdots ,N\}$, let $M_i$ be a complete simply connected Riemannian manifold with sectional curvature lower bounded by $\kappa_i \leq 0$. We define $M:=M_1 \times \cdots \times M_N$. Then, we assume
\begin{description}
\item[(A1)] For all $i \in \{1,2, \cdots ,N\}$, let $X_i \subset M_i$ be a bounded, closed, geodesically convex set \footnote{The closedness and geodesical convexity of $X_i$ imply the uniqueness and existence of $\Pi_{X_i}(x^i)$ \cite[Proposition 2.4]{chong2010iterative} \cite[Theorem 1]{walter1974metric}.} and $X := X_1 \times \cdots \times X_N$. In addition, $X_i \subset M_i$ has a diameter bounded by $D$; i.e., there exists a positive real number $D$ such that
\begin{align*}
\max_{i\in\{1,2, \cdots ,N\}}\sup\{d^i(x^i,y^i):x^i,y^i \in X_i\} \leq D,
\end{align*}
where $d^i(\cdot ,\cdot)$ denotes the distance function of $M_i$;
\item[(A2)] A smooth function $f_t:M \rightarrow \real$ is geodesically convex, where $t$ is a random variable whose probability distribution is a uniform distribution and supported on a set $\mathcal{T} := \{1, 2, \cdots ,T\}$. The function $f$ is defined for all $x \in M$, by $f(x):=\expect{}{f_t(x)}=(1/T)\sum_{t=1}^Tf_t(x)$.
\end{description}
\end{assumption}

Note that, when we define a positive number $G$ as
\begin{align*}
G:=\sup_{t \in \mathcal{T},x \in X}\norm{\grad f_t(x)}{x},
\end{align*}
we find that $G < \infty$ from Assumption \ref{asm:main} (A1). The following is the main problem considered here \cite[Section 4]{becigneul2018riemannian}:
\begin{problem}~\label{pbl:main}
Suppose that Assumption \ref{asm:main} holds. Then, we have
\begin{align*}
x_\ast \in X_\ast := \left\{ x_\ast \in X : f(x_\ast) = \inf_{x \in X}f(x) \right\}.
\end{align*}
\end{problem}

\subsection{Background and motivation}
\label{subsec:2.2}
Euclidean adaptive optimization algorithms, such as AdaGrad, Adam, and AMSGrad, are powerful tools for training deep neural networks. However, there are optimization problems on Riemannian manifolds in machine learning \cite{nickel2017poincare} that cannot be solved by Euclidean adaptive optimization algorithms. Accordingly, useful algorithms, such as RSGD \cite[Section 2]{bonnabel2013stochastic}, RAdaGrad \cite[Section 3.2]{becigneul2018riemannian}, RAdam \cite[Section 4]{becigneul2018riemannian}, and RAMSGrad \cite[Figure 1(a)]{becigneul2018riemannian}, have been developed to solve Riemannian optimization problems. The algorithms, such as RAdaGrad, RAdam, and RAMSGrad, are based on the Euclidean adaptive optimization algorithms, AdaGrad, Adam, and AMSGrad. Hence, we can expect that the corresponding Riemannian adaptive optimization algorithms perform better than RSGD, the most basic Riemannian stochastic optimization algorithm. In fact, the numerical comparisons in \cite{becigneul2018riemannian} showed that Riemannian adaptive optimization algorithms are superior for the task of embedding the WordNet taxonomy in the Poincar\'e ball.

Although Riemannian adaptive optimization algorithms have been shown to be useful for Riemannian optimization in machine learning, we have two motivations related to the previous results in \cite{becigneul2018riemannian}. The first motivation is to identify whether or not RAMSGrad can solve directly the Riemannian optimization problem. This is because the previous results in \cite{becigneul2018riemannian} only showed that RAMSGrad performs a regret minimization, which does not lead to Riemannian minimization (see Subsection \ref{subsec:3.2} for details). The second motivation is to identify whether or not RAMSGrad is applicable to significant problems in fields such as natural language processing and principal component analysis. This is because the previous results in \cite{becigneul2018riemannian} only gave a convergence analysis of RAMSGrad with a diminishing learning rate. Diminishing learning rates are approximately zero after a large number of iterations, which implies that algorithms with them are not implementable in practice. In contrast, a constant learning rate does not cause this problem (see Subsection \ref{subsec:3.2} for details).

Therefore, our goal is to devise an algorithm that overcomes the above issues. In particular, the following section proposes an adaptive optimization algorithm (Algorithm \ref{RAMSGrad}) with a constant learning rate that can solve Problem \ref{pbl:main} directly (Corollary \ref{cor:CLR}).

\section{Riemannian Adaptive Optimization Algorithm}
\label{sec:ALROA}
\subsection{Proposed algorithm and its convergence analysis}
\label{subsec:3.1}
We propose the following algorithm (Algorithm \ref{RAMSGrad}). A small constant $\epsilon > 0$ in the definition of $\hat{v}_n^i$ guarantees that $\sqrt{\hat{v}_n^i} > 0$ (Adam \cite[Algorithm 1]{kingma2015adam} and AMSGrad \cite[Algorithm 2]{reddi2018convergence}, \cite[Algorithm 1]{iiduka2020appropriate} use such a constant in practice).

\begin{algorithm}
\caption{Modified RAMSGrad for solving Problem \ref{pbl:main} \label{RAMSGrad}}
\begin{algorithmic}[1]
\REQUIRE $(\alpha_n)_{n \in \nat} \subset [0,1),(\beta_{1n})_{n \in \nat} \subset [0,1),\beta_2 \in [0,1),\epsilon>0$
\STATE $n \leftarrow 1, x_1 \in X,\tau_0=m_0=0 \in T_{x_0}M,v_0^i, \hat{v}_0^i=0 \in \real$
\LOOP
\STATE $g_{t_n}=(g_{t_n}^i)=\grad{f_{t_n}(x_n)}$
\FOR{$i = 1, 2, \cdots ,N$}
\STATE $m_n^i=\beta_{1n}\tau_{n-1}^i+(1-\beta_{1n})g_{t_n}^i$
\STATE $v_n^i=\beta_2v_{n-1}^i+(1-\beta_2)\norm{g_{t_n}^i}{x_n^i}^2$
\STATE $\hat{v}_n^i=\max\{\hat{v}_{n-1}^i,v_n^i\} + \epsilon$
\STATE $x_{n+1}^i=\Pi_{X_i}\left[\exp_{x_n^i}^i\left(-\alpha_n\dfrac{m_n^i}{\sqrt{\hat{v}_n^i}}\right)\right]$
\STATE $\tau_n^i=\varphi_{x_n^i \rightarrow x_{n+1}^i}^i(m_n^i)$
\ENDFOR
\STATE $n \leftarrow n+1$
\ENDLOOP
\end{algorithmic}
\end{algorithm}

Now, let us compare Algorithm \ref{RAMSGrad} on a Riemannian manifold with AMSGrad in Euclidean space. For simplicity, let us suppose that $M_i = X_i = \mathbb{R}$ ($i=1,2,\ldots,N$). Then, Algorithm \ref{RAMSGrad} defined on $M = \mathbb{R}^N$ is as follows: given $\bm{x}_1 \in \mathbb{R}^N$ and $\bm{m}_0 = \bm{v}_0 = \hat{\bm{v}}_0 = \bm{0} \in \mathbb{R}^N$,
\begin{align*}
&\bm{m}_n = \beta_{1n} \bm{m}_{n-1} + (1 - \beta_{1n})\bm{g}_{t_n},\\
&\bm{v}_n = \beta_2 \bm{v}_{n-1} + (1 - \beta_2) \bm{g}_{t_n} \odot \bm{g}_{t_n},\\
&\hat{\bm{v}}_i = (\hat{v}_n^i)_i = \left(\max \left\{\hat{v}_{n-1}^i,v_n^i \right\} + \epsilon \right)_i,\\
&\bm{x}_{n+1} = \left( x_n^i - \alpha_n \frac{m_n^i}{\sqrt{\hat{v}_n^i}} \right)_i, 
\end{align*}
where $\bm{x} \odot \bm{x} := ({x^i}^2)_i$ for $\bm{x} = (x^i)_i \in \mathbb{R}^N$. This implies Algorithm \ref{RAMSGrad} is an extension of AMSGrad.

Our convergence analysis (Theorem \ref{thm:main}) allows Algorithm 1 to use both constant and diminishing learning rates. Corollaries \ref{cor:CLR}, and \ref{cor:DLR} are convergence analyses of Algorithm 1 with constant and diminishing learning rates, respectively.

\begin{theorem}~\label{thm:main}
Suppose that Assumption \ref{asm:main} holds. Let $(x_n)_{n \in \nat}$ and $(\hat{v}_n)_{n \in \nat}$ be the sequences generated by Algorithm 1. We assume $\beta_{1n} \leq \beta_{1,n-1}$ for all $n \in \nat$, and $(\alpha_n)_{n \in \nat}$ is a sequence of positive learning rates, which satisfies $\alpha_{n}(1-\beta_{1n}) \leq \alpha_{n-1}(1-\beta_{1,n-1})$ for all $n \in \nat$. We define $G:=\max_{t \in \mathcal{T},x \in X}\norm{\grad f_t(x)}{x}$. Then, for all $x_\ast \in X_\ast$,
\begin{align}
\begin{split}~\label{eq:main}
&\expect{}{\frac{1}{n}\sum_{k=1}^nf(x_k) - f(x_\ast)} \\
&\leq \dfrac{NGD^2}{2(1-\beta_{11})}\frac{1}{n\alpha_n} + \dfrac{G^2}{2\sqrt{\epsilon}(1-\beta_{11})}\sum_{i=1}^N\zeta(\kappa_i,D)\frac{1}{n}\sum_{k=1}^n\alpha_k \\
&\quad + \frac{NGD}{1-\beta_{11}}\frac{1}{n}\sum_{k=1}^n\beta_{1k},
\end{split}
\end{align}
where $\zeta(\kappa_i,D)$ is defined as in Lemma \ref{lem:CiAs}.
\end{theorem}

\begin{proof}
See Appendix \ref{sec:apendpt}.
\end{proof}

\begin{corollary}
[Constant learning rate]~\label{cor:CLR} Suppose that the assumptions in Theorem \ref{thm:main} hold, $\alpha_n := \alpha > 0$, and $\beta_{1n}:=\beta \in [0,1)$. Then, Algorithm 1 satisfies, for all $x_\ast \in X_\ast$,
\begin{align*}
\expect{}{\frac{1}{n}\sum_{k=1}^nf(x_k)-f(x_\ast)} \leq \mathcal{O}\left(\frac{1}{n}\right) + C_1\alpha + C_2\beta,
\end{align*}
where $C_1,C_2>0$ are constants.
\end{corollary}
\begin{proof}
See Appendix \ref{sec:apendpc}.
\end{proof}

\begin{corollary}
[Diminishing learning rate]~\label{cor:DLR} Suppose that the assumptions in Theorem \ref{thm:main} hold, $\alpha_n := 1 / n^\eta$, where $\eta \in [1/2,1)$, and $\sum_{k=1}^\infty\beta_{1k} < \infty$ \footnote{ $\alpha_n := 1/n^\eta$ $(\eta \in [1/2,1))$, and $\beta_{1n}=\lambda^n$ ($\lambda \in [0,1)$) satisfy $\sum_{k=1}^\infty\beta_{1k}<\infty$, $\beta_{1n} \leq \beta_{1,n-1}$, and $\alpha_n(1-\beta_{1n}) \leq \alpha_{n-1}(1-\beta_{1,n-1})$ $(n \in \nat)$.}. Then, Algorithm 1 satisfies, for all $x_\ast \in X_\ast$,
\begin{align*}
\expect{}{\frac{1}{n}\sum_{k=1}^nf(x_k)-f(x_\ast)} = \mathcal{O}\left(\dfrac{1}{n^{1-\eta}}\right).
\end{align*}
\end{corollary}
\begin{proof}
See Appendix \ref{sec:apendpc}.
\end{proof}

\subsection{Comparison of Algorithm \ref{RAMSGrad} with the existing algorithms}
\label{subsec:3.2}
Algorithm 1 with $n=t\in\mathcal{T}$ coincides with RAMSGrad \cite[Figure 1(a)]{becigneul2018riemannian}. In \cite{becigneul2018riemannian}, B{\'e}cigneul and Ganea used ``regret" to guarantee the convergence of RAMSGrad. The regret at the end of $T$ iterations is defined as
\begin{align*}
R_T:=\sum_{t \in \mathcal{T}}f_t(x_t) - f_*,
\end{align*}
where $(f_t)_{t \in \mathcal{T}}$ is a family of differentiable, geodesically convex functions from $M$ to $\real$, $f_* := \min_{x\in X} \sum_{t \in \mathcal{T}} f_t(x)$, and $(x_t)_{t \in \mathcal{T}}$ is the sequence generated by RAMSGrad. They proved the following theorem \cite[Theorem 1]{becigneul2018riemannian}:

\begin{theorem}
[Convergence of RAMSGrad]~\label{thm:old} Suppose that Assumption \ref{asm:main} (A1) holds and that $f_t$ is smooth and geodesically convex for all $t \in \mathcal{T}$. Let $(x_t)_{t \in \mathcal{T}}$ and $(\hat{v}_t)_{t \in \mathcal{T}}$ be the sequences obtained from RAMSGrad, $\alpha_t=\alpha / \sqrt{t}$, $\beta_1=\beta_{11}$, $\beta_{1k} \leq \beta_{1}$ for all $t \in \mathcal{T}$, $\alpha > 0$, and $\gamma := \beta_1 / \sqrt{\beta_2} < 1$. We then have:
\begin{align*}
&R_T \leq \frac{\sqrt{T}D^2}{2\alpha(1-\beta_1)}\sum_{i=1}^N\sqrt{\hat{v}_T^i} + \frac{D^2}{2(1-\beta_1)}\sum_{i=1}^N\sum_{t=1}^T\beta_{1t}\dfrac{\sqrt{\hat{v}_t^i}}{\alpha_t} \\
&+\frac{\alpha\sqrt{1+\log{T}}}{(1-\beta_1)^2(1-\gamma)\sqrt{1-\beta_2}}\sum_{i=1}^N\frac{\zeta(\kappa_i,D)+1}{2}\sqrt{\sum_{t=1}^T\norm{g_t^i}{x_t^i}^2}.
\end{align*}
\end{theorem}

Note that Theorem \ref{thm:old} asserts the regret generated by RAMSGrad has an upper bound. We should also note that regret minimization does not always lead to solutions of Problem \ref{pbl:main}. This is because, even if $(x_t)_{t \in \mathcal{T}}$ satisfies, for a sufficiently large number $T$, 
\begin{align*}
R_T = \sum_{t \in \mathcal{T}}f_t(x_t) - f_* \approx 0,
\end{align*}
and we do {\em not} have that
\begin{align*}
Tf(x_T) - f_* = \sum_{t \in \mathcal{T}}f_t(x_T) - f_* \approx 0.
\end{align*} 
Accordingly, Theorem \ref{thm:old} does not guarantee that the output $x_T$ generated by RAMSGrad approximates the solution of Problem \ref{pbl:main}. Additionally, Theorem \ref{thm:old} assumes a diminishing learning rate $\alpha_t$ and does not assert anything about a constant learning rate.

Meanwhile, Corollary \ref{cor:CLR} implies that, if we use sufficiently small constant learning rates $\alpha$ and $\beta$, then Algorithm \ref{RAMSGrad} satisfies 
\begin{align*}
\expect{}{\frac{1}{n}\sum_{k=1}^nf(x_k)-f_*} 
\leq \mathcal{O}\left(\frac{1}{n}\right) + C_1 \alpha + C_2 \beta
\approx \mathcal{O}\left(\frac{1}{n}\right),
\end{align*}
which implies that Algorithm \ref{RAMSGrad} approximates the solution of Problem \ref{pbl:main} in the sense of the mean value of $f(x_k)$. Although Theorem \ref{thm:old} can only use diminishing learning rates such that $\alpha_t := \alpha / \sqrt{t}$, Corollary \ref{cor:CLR} guarantees that Algorithm 1 with a constant learning rate can solve Problem \ref{pbl:main}.

Corollary \ref{cor:DLR} implies that Algorithm 1 with a diminishing learning rate can solve Problem \ref{pbl:main} in the sense that
\begin{align*}
\expect{}{\frac{1}{n}\sum_{k=1}^nf(x_k)-f_*} = \mathcal{O}\left(\dfrac{1}{n^{1-\eta}}\right),
\end{align*}
while Theorem \ref{thm:old} implies that RAMSGrad only minimizes the regret in the sense of the existence of a positive real number $C$ such that
\begin{align*}
\frac{R_T}{T} \leq C\sqrt{\frac{1+\log{T}}{T}}.
\end{align*}
Additionally, Theorem \ref{thm:old} implies RAMSGrad only works in the case where $\eta = 1/2$, but Corollary \ref{cor:DLR} implies Algorithm \ref{RAMSGrad} works for a wider range of $\eta$.

The advantage of Corollary \ref{cor:DLR} over Corollary \ref{cor:CLR} is that using a diminishing learning rate is a robust way to solve Problem \ref{pbl:main}. However, it is possible that Algorithm \ref{RAMSGrad} with a diminishing learning rate does not work for a sufficiently large number $S$ of iterations, because step 8 in Algorithm \ref{RAMSGrad} with $\alpha_S \approx 0$ satisfies
\begin{align*}
x_{S+1}^i=\Pi_{X_i}\left[\exp_{x_S^i}^i\left(-\alpha_S \dfrac{m_S^i}{\sqrt{\hat{v}_S^i}}\right)\right]
\approx
x_S^i.
\end{align*}
Such a trend was observed in \cite{iiduka2020appropriate}. The numerical results in \cite{iiduka2020appropriate} showed that Euclidean adaptive optimization algorithms, such as Adam and AMSGrad, with constant learning rates (e.g., $\alpha_n = \beta_{1n} = 10^{-3}$) perform better than those with diminishing ones in terms of both the training loss and accuracy score. Moreover, we can see that the Euclidean adaptive optimization algorithms in \texttt{torch.optim}\footnote{\url{https://pytorch.org/docs/stable/optim.html}} use constant learning rates, such as $\alpha_n = 10^{-3}$ and $\beta_{1n} = 0.9, 0.999$.

According to \cite[Section 2]{bonnabel2013stochastic} and \cite[Section 5]{becigneul2018riemannian}, useful constant learning rates in RAMSGrad are $\alpha_n = 0.3, 0.1$ and $\beta_{1n} = 0.9$. Meanwhile, Corollary \ref{cor:CLR} indicates that using a small constant learning rate $\beta_{1n}$ would be good for solving Problem \ref{pbl:main}. Accordingly, the next section numerically compares the behavior of Algorithm \ref{RAMSGrad} with $\beta_{1n} = 0.9$ with one with $\beta_{1n} = 0.001 < 0.9$. Corollary \ref{cor:DLR} (see also Theorem \ref{thm:old}) indicates that Algorithm \ref{RAMSGrad} should use diminishing learning rates such that $\alpha_n = \mathcal{O}(1/\sqrt{n})$ and $\beta_{1n} = \lambda^n$ ($\lambda \in [0,1)$). The next section uses diminishing learning rates to compare fairly the behaviors of Algorithm \ref{RAMSGrad} with constant learning rates (see Section \ref{sec:NE} for details).

\section{Numerical Experiments}
~\label{sec:NE}
We numerically compared the following Riemannian stochastic optimization algorithms: RSGD \cite[Section 2]{bonnabel2013stochastic}, RAdaGrad \cite[Section 3.2]{becigneul2018riemannian}, RAdam \cite[Section 4]{becigneul2018riemannian}, and Algorithm \ref{RAMSGrad} (modified RAMSGrad). RAdam is obtained by removing the $\max$ operation in Algorithm \ref{RAMSGrad}, i.e., replacing $\hat{v}_n^i=\max\{\hat{v}_{n-1}^i,v_n^i\} + \epsilon$ with $\hat{v}_n^i = v_n^i + \epsilon$ (see \cite{becigneul2018riemannian}). Our experiments were conducted on a fast scalar computation server\footnote{\url{https://www.meiji.ac.jp/isys/hpc/ia.html}} at Meiji University. The environment has two Intel(R) Xeon(R) Gold 6148 (2.4 GHz, 20 cores) CPUs, an NVIDIA Tesla V100 (16GB, 900Gbps) GPU and a Red Hat Enterprise Linux 7.6 operating system.

\subsection{Poincar{\'e} embeddings}
In \cite{nickel2017poincare}, Nickel and Kiela developed Poincar{\'e} embeddings. Before describing the numerical experiments, we will review the fundamentals of hyperbolic geometry (see \cite{nickel2017poincare, ungar2008gyrovector, ganea2018hyperbolic, becigneul2018riemannian}). $\mathcal{B}^d:=\{x \in \real^d:\norm{x}{}<1\}$ denotes the open $d$-dimensional unit ball, where $\norm{\cdot}{}$ denotes the Euclidean norm. The Poincar{\'e} ball model of hyperbolic space $(\mathcal{B}^d,\rho)$ is defined by a manifold $\mathcal{B}^d$ equipped with the following Riemannian metric:
\begin{align*}
\rho_x := \frac{4}{(1-\norm{x}{}^2)^2}\rho^E_x,
\end{align*}
where $x \in \mathcal{B}^d$, and $\rho^E_x$ denotes the Euclidean metric tensor. The Riemannian manifold $(\mathcal{B}^d,\rho)$ has a constant sectional curvature, $-1$. We define M\"{o}bius addition \cite[Definition 1.10]{ungar2008gyrovector} of $x$ and $y$ in $\mathcal{B}^d$ as
\begin{align*}
x \oplus_M y := \frac{(1+2\ip{x}{y}{}+\norm{y}{}^2)x + (1-\norm{x}{}^2)y}{1+2 \ip{x}{y}{} + \norm{x}{}^2\norm{y}{}^2},
\end{align*}
where $\ip{\cdot}{\cdot}{} := \rho^E(\cdot ,\cdot)$. Moreover, $\ominus_M x$ denotes the left inverse \cite[Definition 1.7]{ungar2008gyrovector} of $x \in \mathcal{B}^d$, and the M{\"o}bius gyrations \cite[Definition 1.11]{ungar2008gyrovector} of $\mathcal{B}^d$ are defined as
\begin{align*}
\textrm{gyr}[x,y]z := \ominus_M (x \oplus_M y) \oplus_M \{x \oplus_M (y \oplus_M z)\},
\end{align*}
for all $x,y,z \in \mathcal{B}^d$.

In accordance with the above statements, the induced distance function on $(\mathcal{B}^d,\rho)$ (see \cite[Eq. (6)]{ganea2018hyperbolic}) is defined for all $x,y \in \mathcal{B}^d$, by
\begin{align}~\label{eq:dist}
d(x,y)=2\tanh^{-1}\left(\norm{(-x) \oplus_M y}{}\right).
\end{align}
The exponential map on $(\mathcal{B}^d,\rho)$ (see \cite[Lemma 2]{ganea2018hyperbolic}) is expressed as follows: for $x\in\mathcal{B}^d$ and $\xi \neq 0 \in T_x\mathcal{B}^d$,
\begin{align*}
\exp_{x}(\xi)=x \oplus_M \left\{ \tanh\left(\frac{\norm{\xi}{}}{1-\norm{x}{}^2}\right) \right\}\frac{\xi}{\norm{\xi}{}},
\end{align*}
and, for $x\in\mathcal{B}^d$ and $0 \in T_x\mathcal{B}^d$,
\begin{align*}
\exp_{x}(0)=x.
\end{align*}
Parallel transport of $(\mathcal{B}^d,\rho)$ (see \cite[Section 5]{becigneul2018riemannian}) along the unique geodesic from $x$ to $y$ is given by
\begin{align*}
\varphi_{x \rightarrow y}(\xi) = \frac{1-\norm{y}{}^2}{1-\norm{x}{}^2}\textrm{gyr}[y,-x]\xi .
\end{align*}
The Riemannian gradient on $(\mathcal{B}^d,\rho)$ (see \cite[Section 5]{becigneul2018riemannian}) is expressed in terms of rescaled Euclidean gradients, i.e., for $x \in \mathcal{B}^d$, and the smooth function $f:\mathcal{B}^d \rightarrow \real$,
\begin{align*}
\grad{f(x)} = \frac{(1-\norm{x}{}^2)^2}{4}\nabla^Ef(x),
\end{align*}
where $\nabla^Ef(x)$ denotes the Euclidean gradient of $f$.

To compute the Poincar{\'e} embeddings for a set of $N$ symbols by finding the embeddings $\Theta = \{u_i\}_{i=1}^N$, where $u_i \in \mathcal{B}^d$, we solve the following optimization problem: given $\mathcal{L}:\mathcal{B}^d \times \cdots \times \mathcal{B}^d \rightarrow \real$,
\begin{align}~\label{eq:PEproblem}
\textrm{minimize }\mathcal{L}(\Theta) \quad \textrm{subject to } u_i \in \mathcal{B}^d.
\end{align}

The transitive closure of the WordNet mammals subtree consists of 1,180 nouns and 6,450 hypernymy Is-A relations. Let $\mathcal{D} = \{(u, v)\}$ be the set of observed hypernymy relations between noun pairs. We minimize a loss function defined by
\begin{align}~\label{eq:lossfunction}
\mathcal{L}(\Theta)=\sum_{(u,v)\in\mathcal{D}}\log{\dfrac{e^{-d(u,v)}}{\sum_{v^\prime \in \mathcal{N}(u)}e^{-d(u,v^\prime)}}},
\end{align}
where $d(u,v)$ defined by \eqref{eq:dist} is the corresponding distance of the relation $(u,v) \in \mathcal{D}$, and $\mathcal{N}(u) = \{v^\prime : (u,v^\prime) \not\in \mathcal{D}\} \cup \{v\}$ is the set of negative examples for $u$ including $v$ (see \cite{becigneul2018riemannian, nickel2017poincare}). We embed the transitive closure of the WordNet mammals subtree into a $5$-dimensional Poincar{\'e} ball $(\mathcal{B}^5,\rho)$.

Let us define $M_i:=\mathcal{B}^5$ and $X_i:=\{x\in\mathcal{B}^5:\norm{x}{}\leq 1-10^{-5}\}$, whose projection operator $\Pi_{X_i}:\mathcal{B}^5 \rightarrow X_i$ is computed as
\begin{align*}
\Pi_{X_i}(x):=
\begin{cases}
x & \textrm{if }\norm{x}{} \leq 1-10^{-5} \\
(1-10^{-5})\dfrac{x}{\norm{x}{}} & \textrm{otherwise}
\end{cases}.
\end{align*}
Moreover, the geodesically convex set $X_i$ has a bounded diameter; in fact, let $D$ be the diameter of a closed disk $X_i$, measured by the Riemann metric of $\rho$.

As in \cite{nickel2017poincare}, we will introduce an index for evaluating the embedding. For each observed relation $(u, v) \in \mathcal{D}$, we compute the corresponding distance $d(u, v)$ in the embedding and rank it among the set of negative relations for $u$, i.e., among the set $\{d(u, v^\prime):(u,v^\prime) \not\in \mathcal{D}\}$. In addition, we assume the reconstruction setting (see \cite{nickel2017poincare}); i.e., we evaluate the ranking of all nouns in the dataset. Then, we record the mean rank of $v$ as well as the mean average precision (MAP) of the ranking. Thus, we evaluate the embedding in terms of the loss function values and the MAP rank.

We experimented with a special iteration called the ``burn-in phase" (see \cite[Section 3]{nickel2017poincare}) for the first 20 epochs. During the burn-in phase, the algorithm runs at a reduced learning rate of $1/100$. When we minimized the loss function \eqref{eq:lossfunction}, we randomly sampled 10 negative relations per positive relation. We set $\epsilon = 10^{-8}$ in Algorithm 1.

The experiment used the code of Facebook Research\footnote{\url{https://github.com/facebookresearch/poincare-embeddings}}, and we used the NumPy 1.17.3 package and PyTorch 1.3.0 package.

\subsubsection{Constant learning rate}
First, we compared algorithms with the following ten constant learning rates:
\begin{description}
\item[(CS1)] RSGD: $\alpha_n = 0.3$.
\item[(CS2)] RSGD: $\alpha_n = 0.1$.
\item[(CG1)] RAdaGrad: $\alpha_n=0.3$.
\item[(CG2)] RAdaGrad: $\alpha_n=0.1$.
\item[(CD1)] RAdam: $\alpha_n=0.3$, $\beta_{1n} = 0.9$, $\beta_2 = 0.999$.
\item[(CD2)] RAdam: $\alpha_n=0.1$, $\beta_{1n} = 0.9$, $\beta_2 = 0.999$.
\item[(CA1)] Algorithm 1: $\alpha_n = 0.3$, $\beta_{1n} = 0.9$, $\beta_2 = 0.999$.
\item[(CA2)] Algorithm 1: $\alpha_n = 0.3$, $\beta_{1n} = 0.001$, $\beta_2 = 0.999$.
\item[(CA3)] Algorithm 1: $\alpha_n = 0.1$, $\beta_{1n} = 0.9$, $\beta_2 = 0.999$.
\item[(CA4)] Algorithm 1: $\alpha_n = 0.1$, $\beta_{1n} = 0.001$, $\beta_2 = 0.999$.
\end{description}
The parameter $\alpha_n$ in (CS1) and (CS2) represents the learning rate of RSGD \cite[Section 2]{bonnabel2013stochastic}. The learning rates of (CA1)--(CA4) satisfy the assumptions of Corollary \ref{cor:CLR}. The parameters $\beta_2 = 0.999$ and $\beta_{1n}=0.9$ in (CA1) and (CA3) are used in \cite[Section 5]{becigneul2018riemannian}. We used $\beta_{1n}=0.001$ in (CA2) and (CA4) to compare (CA1) and (CA3) with Algorithm \ref{RAMSGrad} with a small learning rate. Figs. \ref{fig:c-lo-ep}--\ref{fig:c-map-el} show the numerical results. Fig. \ref{fig:c-lo-ep} shows the performances of the algorithms for loss function values defined by \eqref{eq:lossfunction} with respect to the number of epochs, while Fig. \ref{fig:c-lo-el} presents those with respect to the elapsed time. Fig. \ref{fig:c-map-ep} shows the MAP ranks of the embeddings with respect to the number of epochs, while Fig. \ref{fig:c-map-el} presents the MAP ranks with respect to the elapsed time. We can see that Algorithm \ref{RAMSGrad} outperforms RSGD and RAdaGrad in every setting. In particular, Figs. \ref{fig:c-lo-ep}--\ref{fig:c-lo-el} show that the learning outcomes of RSGD fluctuate greatly depending on the learning rate. In contrast, Algorithm \ref{RAMSGrad} and RAdam eventually reduce the loss function the most for any learning rate. Moreover, these figures show that the performance of (CA1) (resp. (CA3)) is comparable to that of (CA2) (resp. (CA4)). Meanwhile, RAdaGrad quickly reduced the objective function value in the early stages; however, it soon stopped learning.

\begin{figure}[!t]
\centering
\includegraphics[width=3.0in]{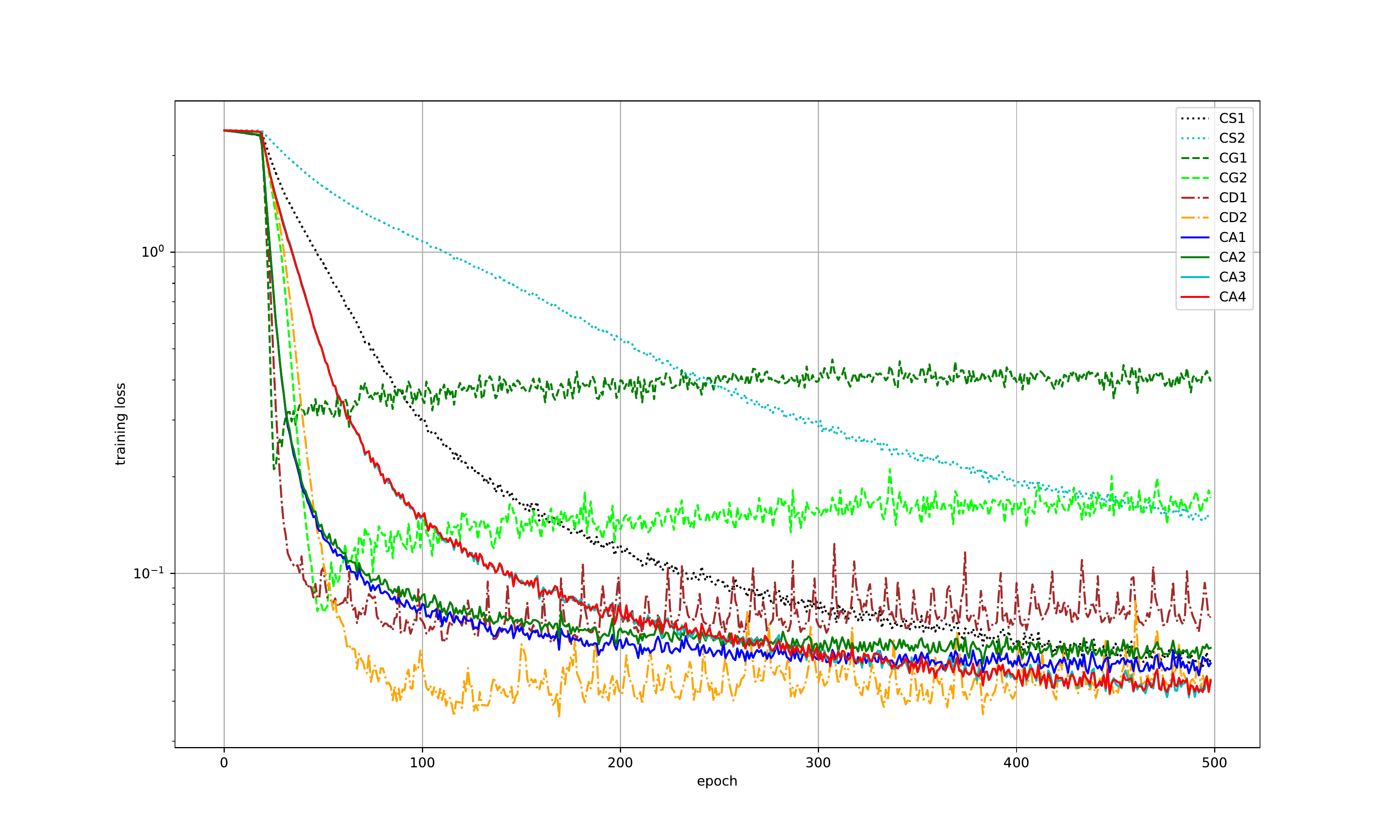}
 \DeclareGraphicsExtensions{.pdf}
\caption{Loss function value versus number of epochs in the case of constant learning rates. \label{fig:c-lo-ep}}
\end{figure}

\begin{figure}[!t]
\centering
\includegraphics[width=3.0in]{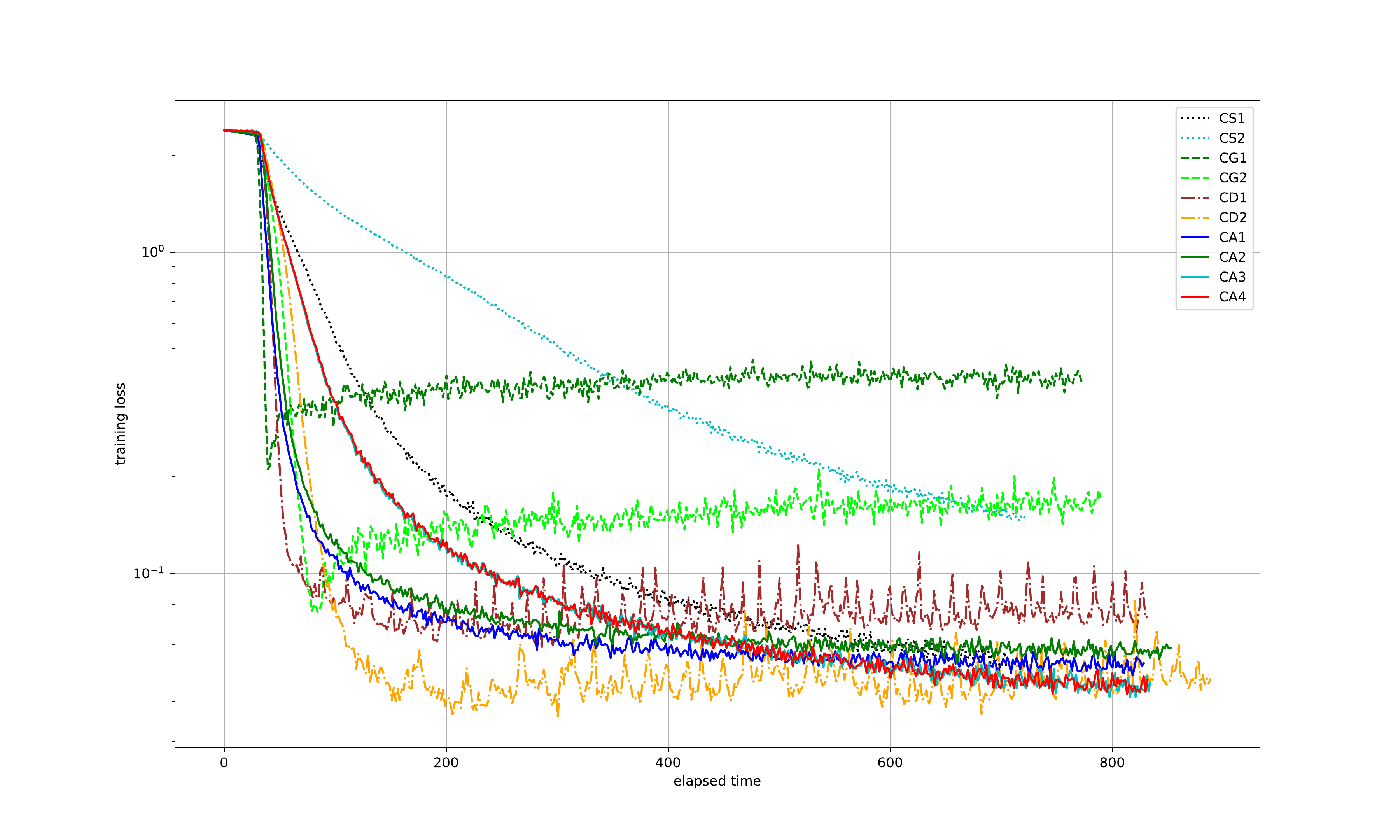}
 \DeclareGraphicsExtensions{.pdf}
\caption{Loss function value versus elapsed time in the case of constant learning rates. \label{fig:c-lo-el}}
\end{figure}

\begin{figure}[!t]
\centering
\includegraphics[width=3.0in]{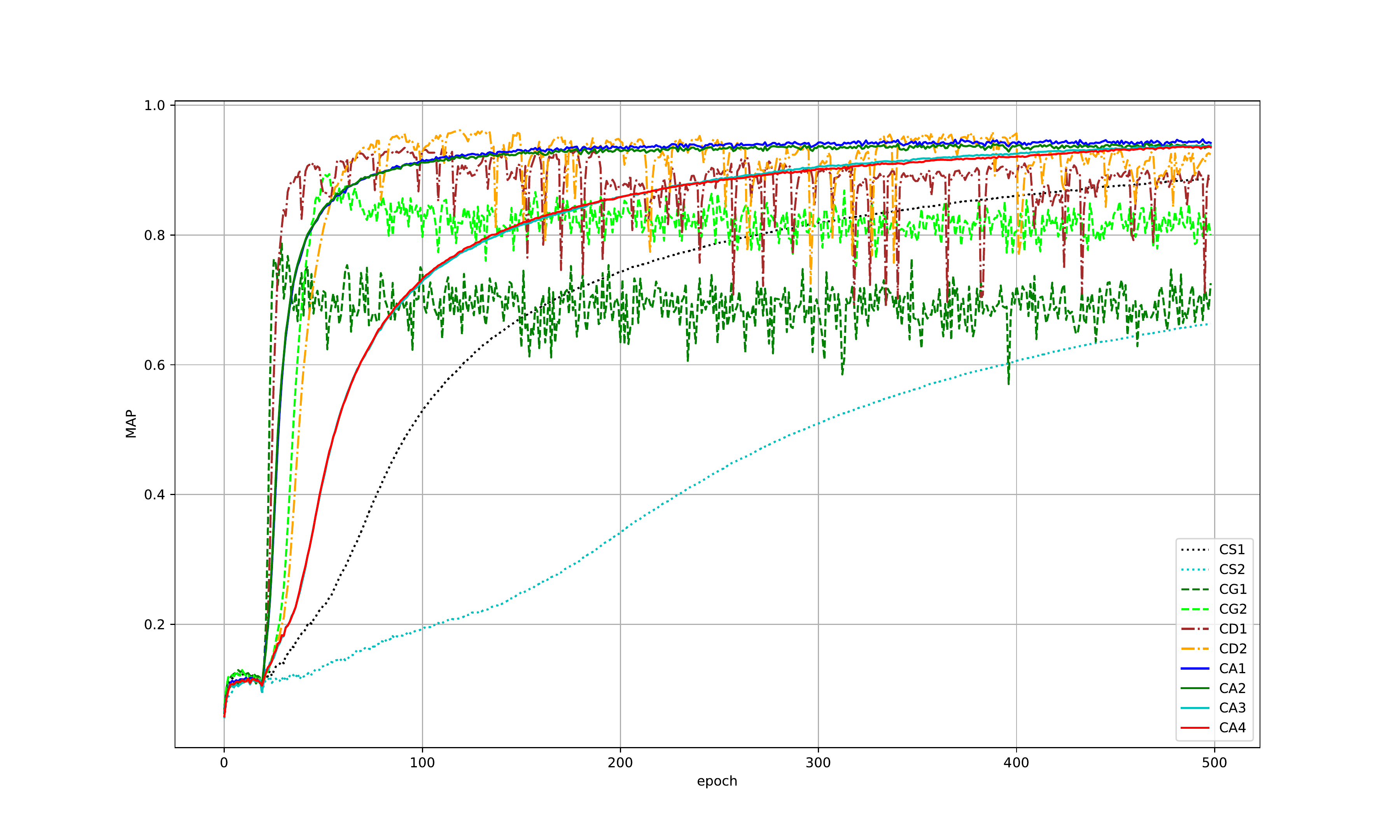}
 \DeclareGraphicsExtensions{.pdf}
\caption{MAP rank versus number of epochs in the case of constant learning rates. \label{fig:c-map-ep}}
\end{figure}

\begin{figure}[!t]
\centering
\includegraphics[width=3.0in]{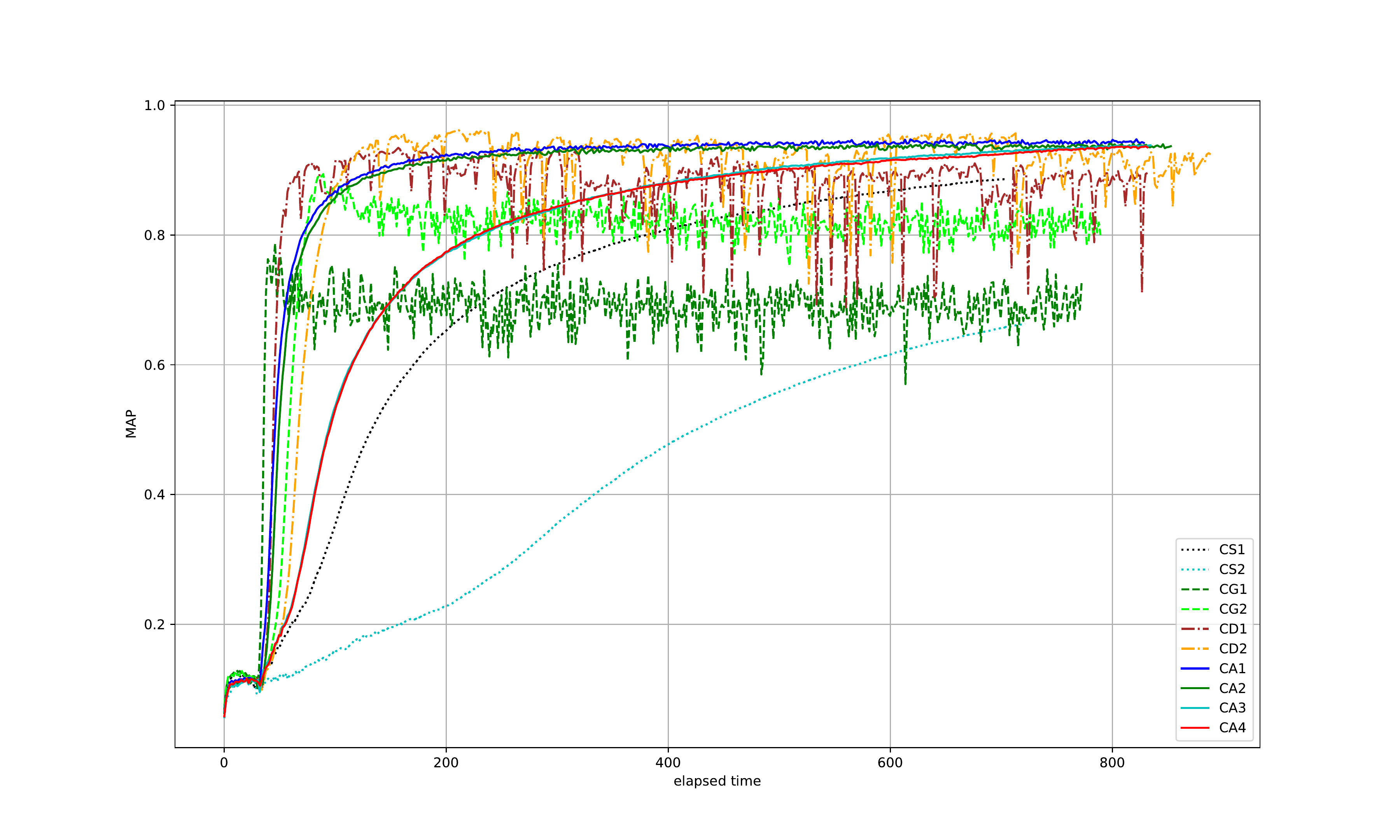}
 \DeclareGraphicsExtensions{.pdf}
\caption{MAP rank versus elapsed time in the case of constant learning rates. \label{fig:c-map-el}}
\end{figure}

\subsubsection{Diminishing learning rate}
Next, we compared algorithms with the following ten diminishing learning rates:
\begin{description}
\item[(DS1)] RSGD: $\alpha_n = 30 / \sqrt{n}$.
\item[(DS2)] RSGD: $\alpha_n = 10 / \sqrt{n}$.
\item[(DG1)] RAdaGrad: $\alpha_n=30 / \sqrt{n}$.
\item[(DG2)] RAdaGrad: $\alpha_n=10 / \sqrt{n}$.
\item[(DD1)] RAdam: $\alpha_n=30 / \sqrt{n}$, $\beta_{1n} = 0.5^n$, $\beta_2 = 0.999$.
\item[(DD2)] RAdam: $\alpha_n=10 / \sqrt{n}$, $\beta_{1n} = 0.5^n$, $\beta_2 = 0.999$.
\item[(DA1)] Algorithm 1: $\alpha_n = 30 / \sqrt{n}$, $\beta_{1n} = 0.5^n$, $\beta_2 = 0.999$.
\item[(DA2)] Algorithm 1: $\alpha_n = 30 / \sqrt{n}$, $\beta_{1n} = 0.9^n$, $\beta_2 = 0.999$.
\item[(DA3)] Algorithm 1: $\alpha_n = 10 / \sqrt{n}$, $\beta_{1n} = 0.5^n$, $\beta_2 = 0.999$.
\item[(DA4)] Algorithm 1: $\alpha_n = 10 / \sqrt{n}$, $\beta_{1n} = 0.9^n$, $\beta_2 = 0.999$.
\end{description}
The learning rates of (DA1)--(DA4) satisfy the assumptions of Corollary \ref{cor:DLR}. We implemented (DA2) and (DA4) to compare them with (CA1) and (CA3). We implemented (DA1) and (DA3) to check how well Algorithm \ref{RAMSGrad} works depending on the choice of $\beta_{1n}$. Figs. \ref{fig:d-lo-ep}--\ref{fig:d-map-el} show the numerical results. Fig. \ref{fig:d-lo-ep} shows the behaviors of the algorithms for loss function values defined by \eqref{eq:lossfunction} with respect to the number of epochs, whereas Fig. \ref{fig:d-lo-el} shows those with respect to the elapsed time. Fig. \ref{fig:d-map-ep} presents the MAP ranks of the embeddings with respect to the number of epochs, while Fig. \ref{fig:d-map-el} shows MAP ranks with respect to the elapsed time. Even in the case of diminishing learning rates, Algorithm \ref{RAMSGrad} outperforms RSGD in every setting. The learning results of RSGD fluctuate greatly depending on the initial learning rate. In particular, (DS2) reduces the loss function more slowly than the other algorithms do. On the other hand, Algorithm \ref{RAMSGrad} and RAdam stably reduce the loss function, regardless of the initial learning rate. Moreover, these figures indicate that (DA2) outperforms (DA1) and that (DA3) performs comparably to (DA4). In addition, RAdaGrad is better or worse than Algorithm \ref{RAMSGrad} depending on how we choose the initial learning rates.

From Figs. 2 and 6, we can see that (CA1) (resp. (CA3)) outperforms (DA2) (resp. (DA4)). The above discussion shows that Algorithm \ref{RAMSGrad} with a constant learning rate is superior to the other algorithms at embedding the WordNet mammals subtree into a Poinca{\'e} ball.

\begin{figure}[!t]
\centering
\includegraphics[width=3.0in]{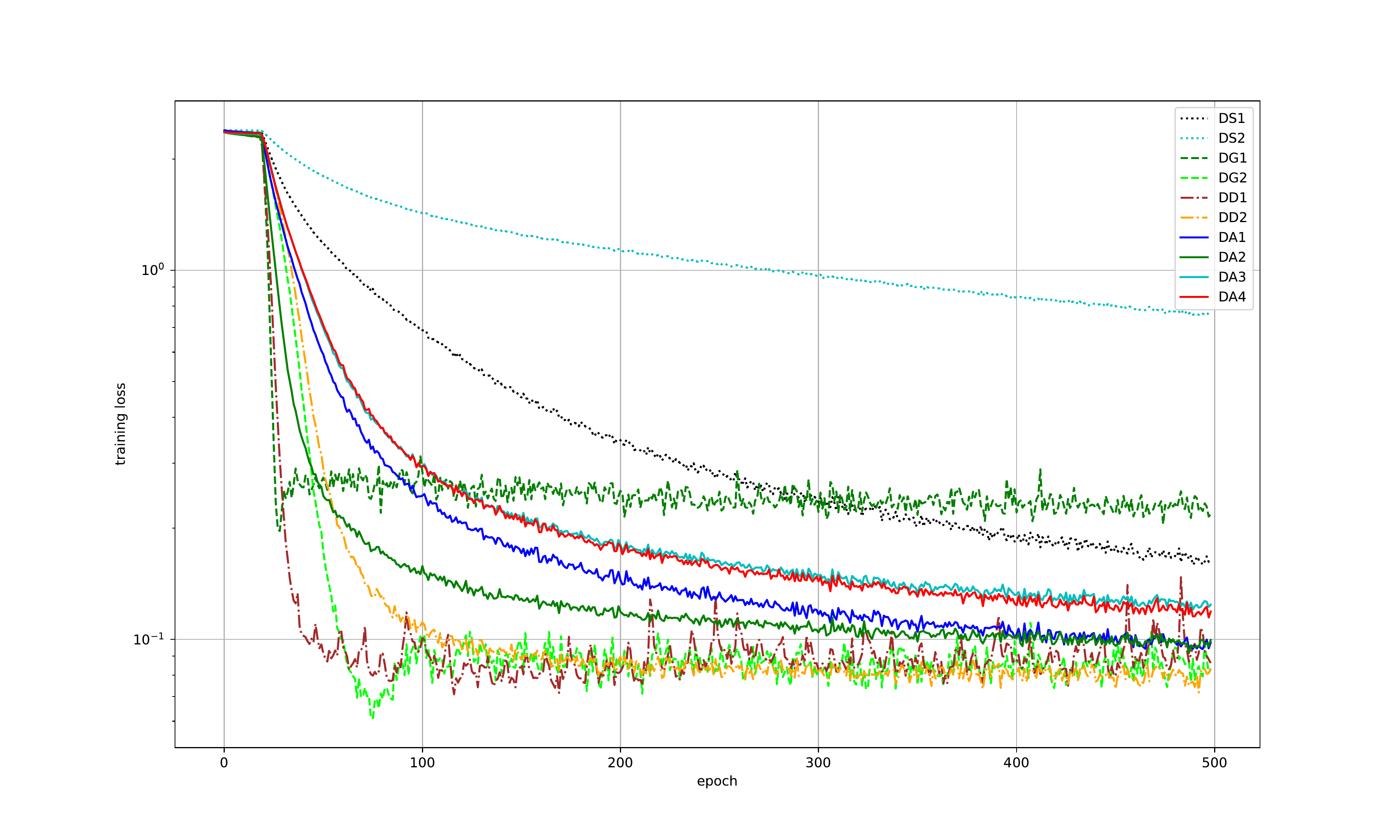}
 \DeclareGraphicsExtensions{.pdf}
\caption{Loss function value versus number of epochs in the case of diminishing learning rates. \label{fig:d-lo-ep}}
\end{figure}

\begin{figure}[!t]
\centering
\includegraphics[width=3.0in]{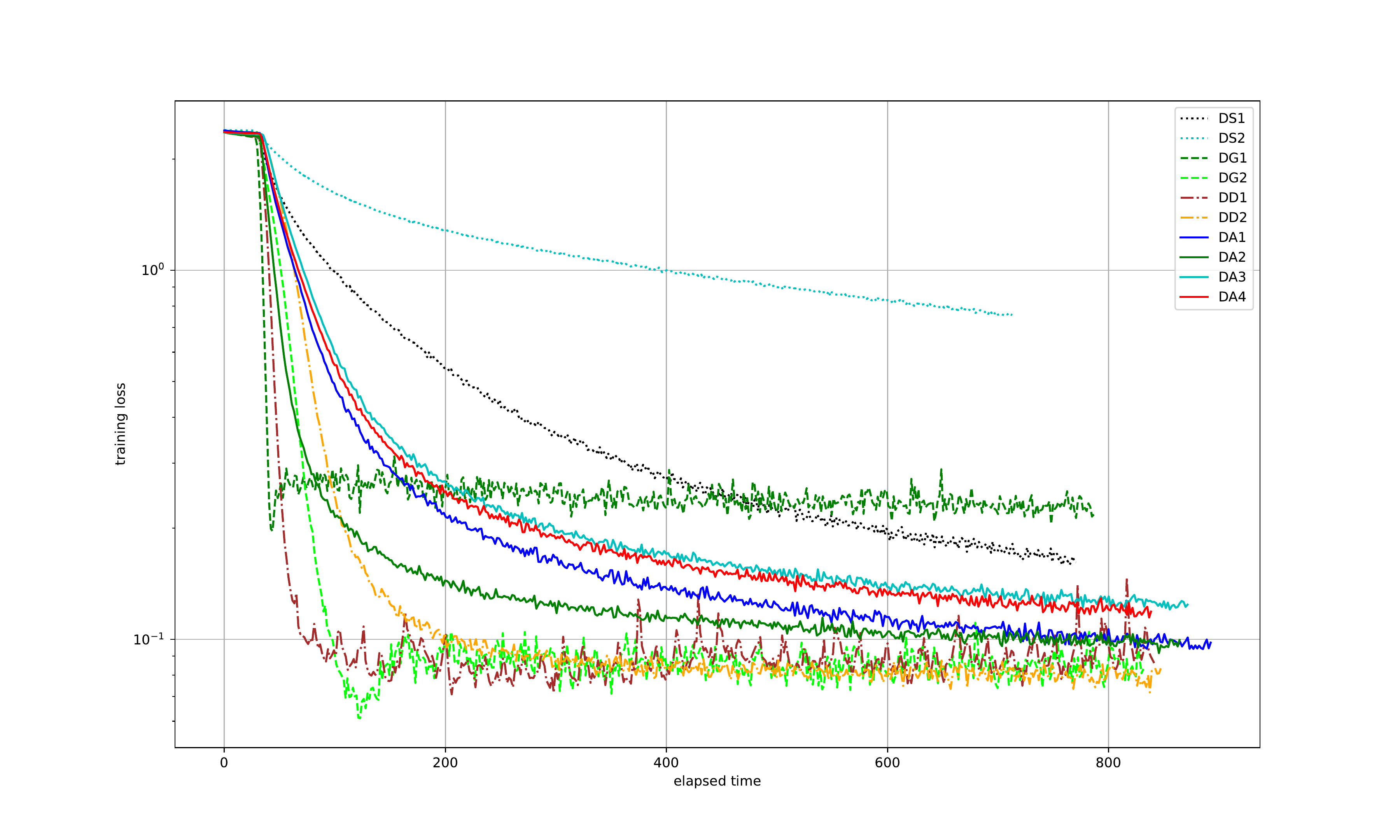}
 \DeclareGraphicsExtensions{.pdf}
\caption{Loss function value versus elapsed time in the case of diminishing learning rates. \label{fig:d-lo-el}}
\end{figure}

\begin{figure}[!t]
\centering
\includegraphics[width=3.0in]{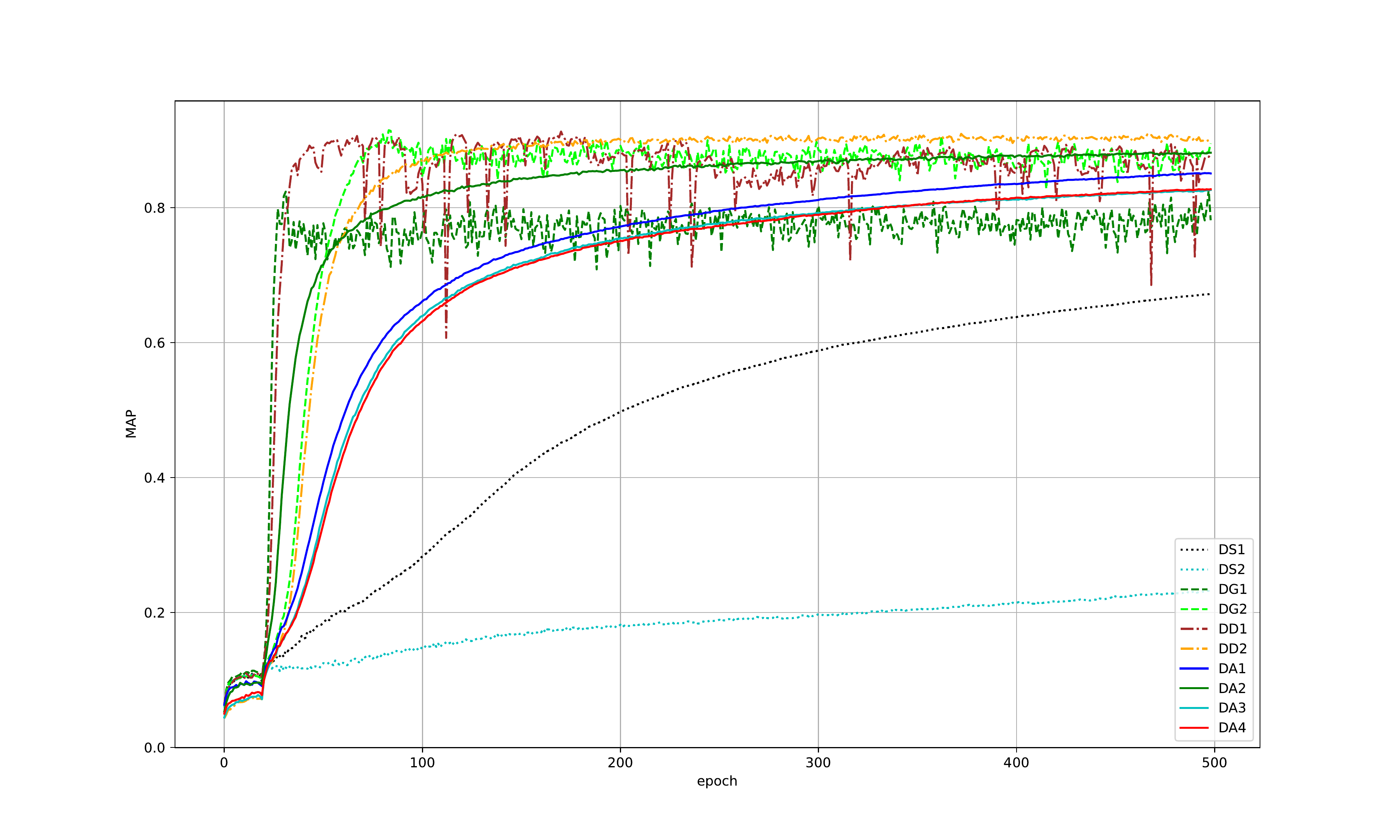}
 \DeclareGraphicsExtensions{.pdf}
\caption{MAP rank versus number of epochs in the case of diminishing learning rates. \label{fig:d-map-ep}}
\end{figure}

\begin{figure}[!t]
\centering
\includegraphics[width=3.0in]{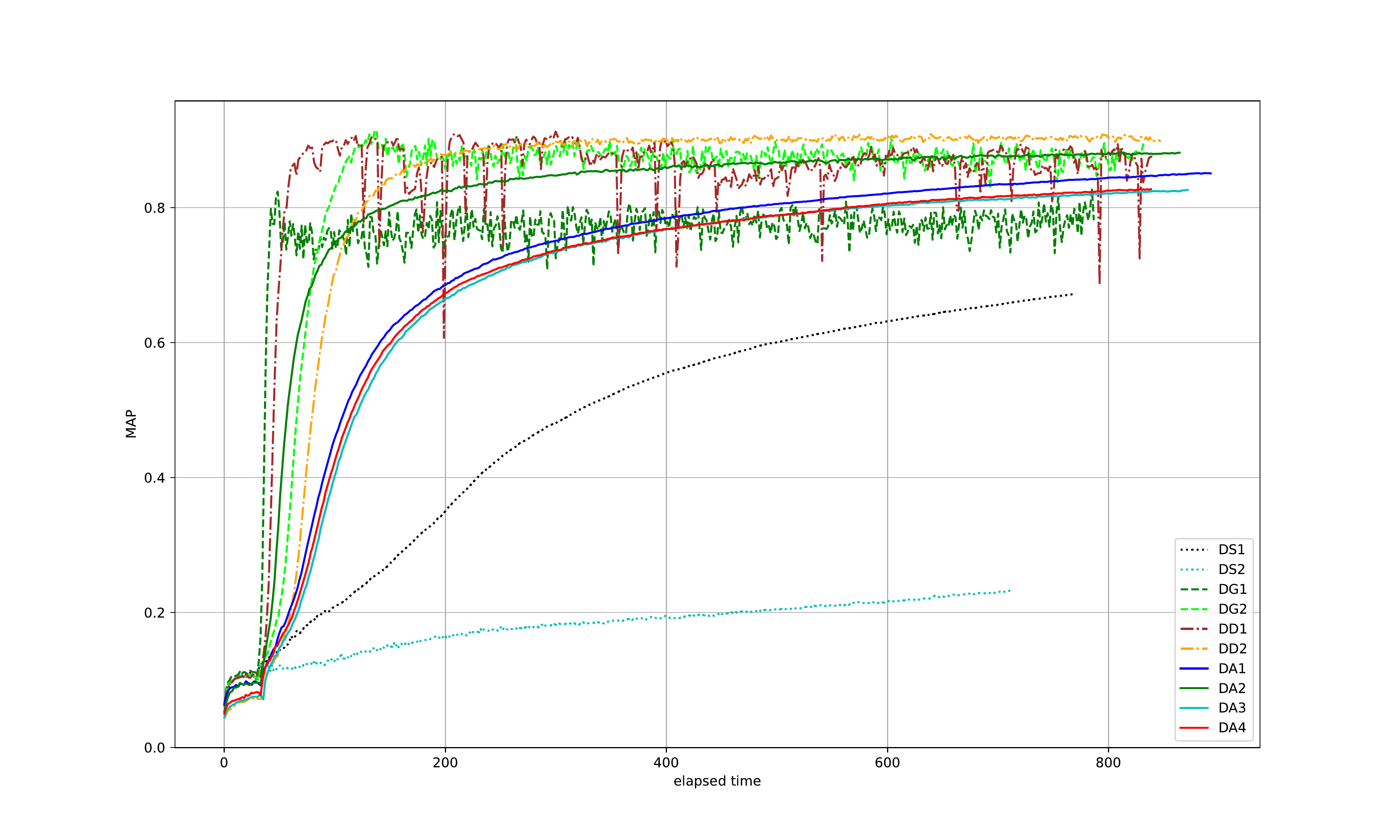}
 \DeclareGraphicsExtensions{.pdf}
\caption{MAP rank versus elapsed time in the case of diminishing learning rates. \label{fig:d-map-el}}
\end{figure}

\subsection{Principal component analysis}
Here, we applied the algorithms to a principal component analysis (PCA) problem. Given $n$ data points $a_1, \cdots ,a_n \in \real^d$, the PCA problem (see \cite{kasai2019riemannian, zhou2019faster}) is formulated as
\begin{align*}
\textrm{minimize }f(U) \quad \textrm{subject to } U \in \stiefel(k,d),
\end{align*}
where
\begin{align*}
f(U):=-\frac{1}{n}\sum_{i=1}^n a_i^\top UU^\top a_i,
\end{align*}
and $\stiefel(k, d) := \{U \in \real^{d \times k} : U^\top U = I_k \}$ denotes the Stiefel manifold. For this problem, we set $N=1$ and $M=\stiefel(k,d)$. Since it is known that parallel transport has no closed-form solution on the Stiefel manifold, we use QR-based retraction and the associated vector transport as an approximation of the exponential map and the parallel transport, respectively (see \cite{absil2008}). The QR-based retraction is defined as
\begin{align*}
R_U(\xi) := \mathrm{qf}(U + \xi),
\end{align*}
where $U \in \stiefel(k,d)$, $\xi \in T_U\stiefel(k,d)$ and $\mathrm{qf}(A)$ denotes the Q factor of the QR decomposition of $A$. Then, the associated vector transport is defined as
\begin{align*}
\mathcal{T}_{U \rightarrow V}(\xi) := \xi - V \mathrm{sym}(V^\top\xi),
\end{align*}
where $U,V \in \stiefel(k,d)$, $\xi \in T_U\stiefel(k,d)$ and $\mathrm{sym}(A):=(A+A^\top)/2$. For this problem, the columns of the optimal solution $U_\ast$ are known to be the top $k$ eigenvectors of the data covariance matrix, which can be estimated using singular value decomposition. The performance of each algorithm in the experiment was judged in terms of the ``optimality gap", that is, $f(U) - f(U_\ast)$. We evaluated the algorithms on the \texttt{MNIST} \footnote{\url{https://keras.io/ja/datasets/}} and \texttt{digits} \footnote{\url{https://scikit-learn.org/stable/auto_examples/datasets/plot_digits_last_image.html}} datasets. The \texttt{MNIST} dataset contains handwritten digits data of 0--9 and has 10000 images of size $28 \times 28$ for testing (see \cite{lecun1998mnist}). For the \texttt{MNIST} dataset, we set $(n,k,d)=(10000,784,10)$. The \texttt{digits} dataset is made up of 1797 $8 \times 8$ handwritten digit images. For the \texttt{digits} dataset, we set $(n,k,d)=(1797,64,8)$.

\subsubsection{Constant learning rate}
First, we compared ten algorithms (CS1)--(CA4) with constant learning rates, same as those used in the experiments on the Poincar{\'e} embeddings. Figs. \ref{fig:mnist-c1}--\ref{fig:mnist-c2} show the numerical results on the \texttt{MNIST} dataset, while Figs. \ref{fig:digits-c1}--\ref{fig:digits-c2} show the numerical results on the \texttt{digits} dataset. Figs. \ref{fig:mnist-c1}--\ref{fig:mnist-c2} indicate that Algorithm \ref{RAMSGrad} performed well in every setting for the \texttt{MNIST} dataset. In particular, the behaviors of (CA2) and (CA4) are the best of all algorithms and the behavior of (CD1) is comparable to them. Moreover, Figs. \ref{fig:digits-c1}--\ref{fig:digits-c2} show that Algorithm \ref{RAMSGrad} outperforms RSGD and RAdaGrad in every setting for the \texttt{digits} dataset. These figures indicate that (CD1), (CA1) and (CA3) eventually made the optimal gap the smallest. Meanwhile, Figs. \ref{fig:mnist-c1}--\ref{fig:digits-c2} show that RAdaGrad often failed to reduce the optimal gap.

\begin{figure}[!t]
\centering
\includegraphics[width=3.0in]{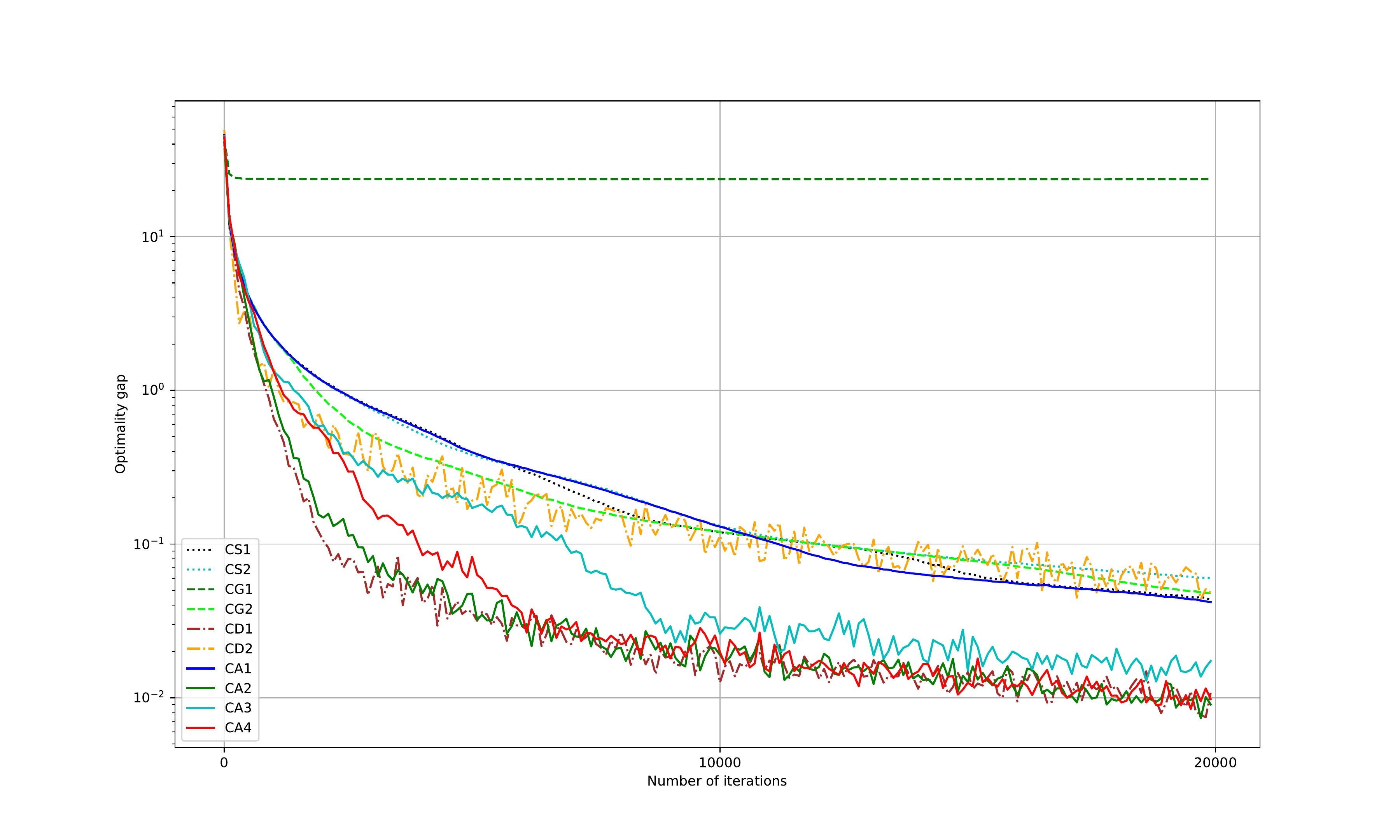}
 \DeclareGraphicsExtensions{.pdf}
\caption{Optimality gap versus number of iterations in the case of constant learning rates for the \texttt{MNIST} dataset \label{fig:mnist-c1}}
\end{figure}

\begin{figure}[!t]
\centering
\includegraphics[width=3.0in]{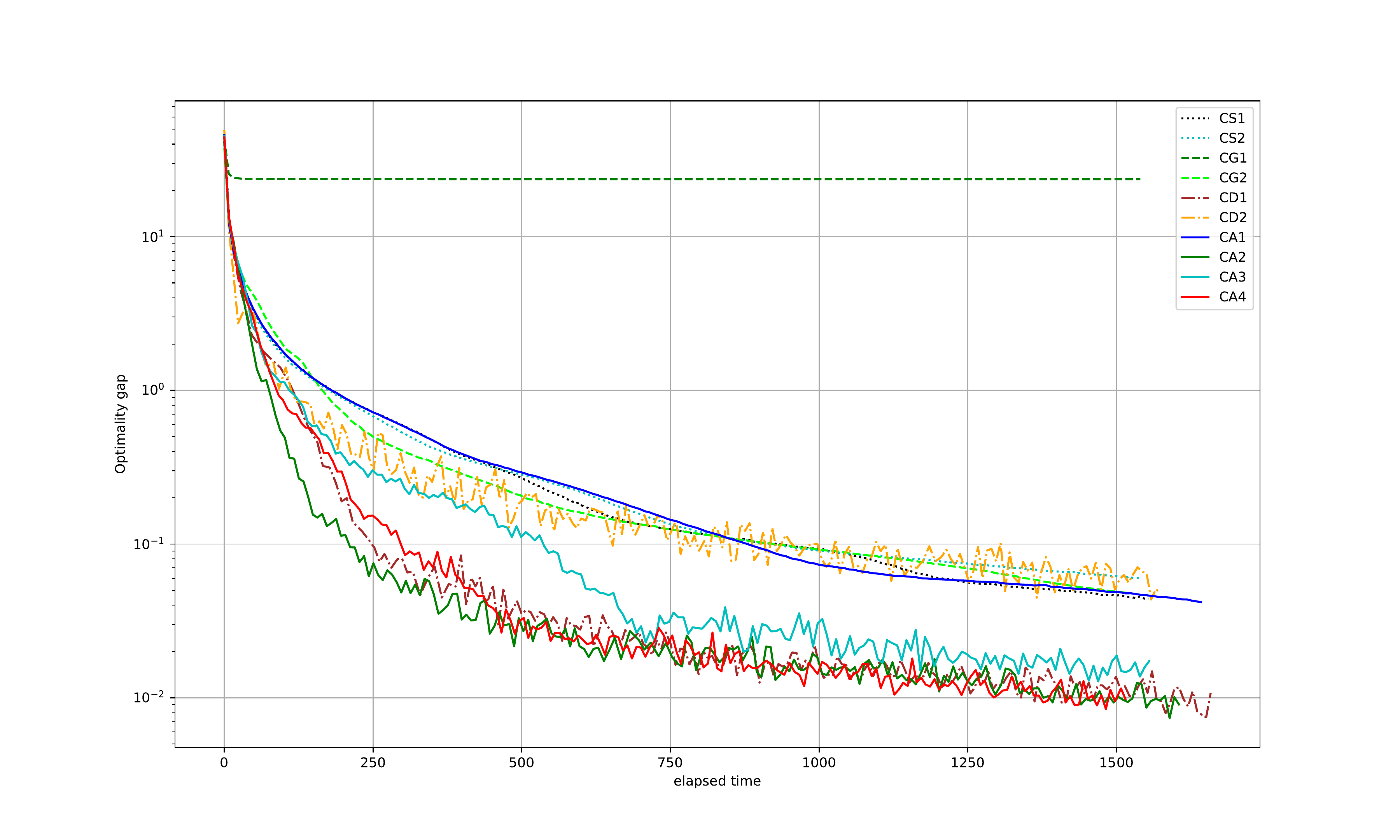}
 \DeclareGraphicsExtensions{.pdf}
\caption{Optimality gap versus elapsed time in the case of constant learning rates for the \texttt{MNIST} dataset \label{fig:mnist-c2}}
\end{figure}

\begin{figure}[!t]
\centering
\includegraphics[width=3.0in]{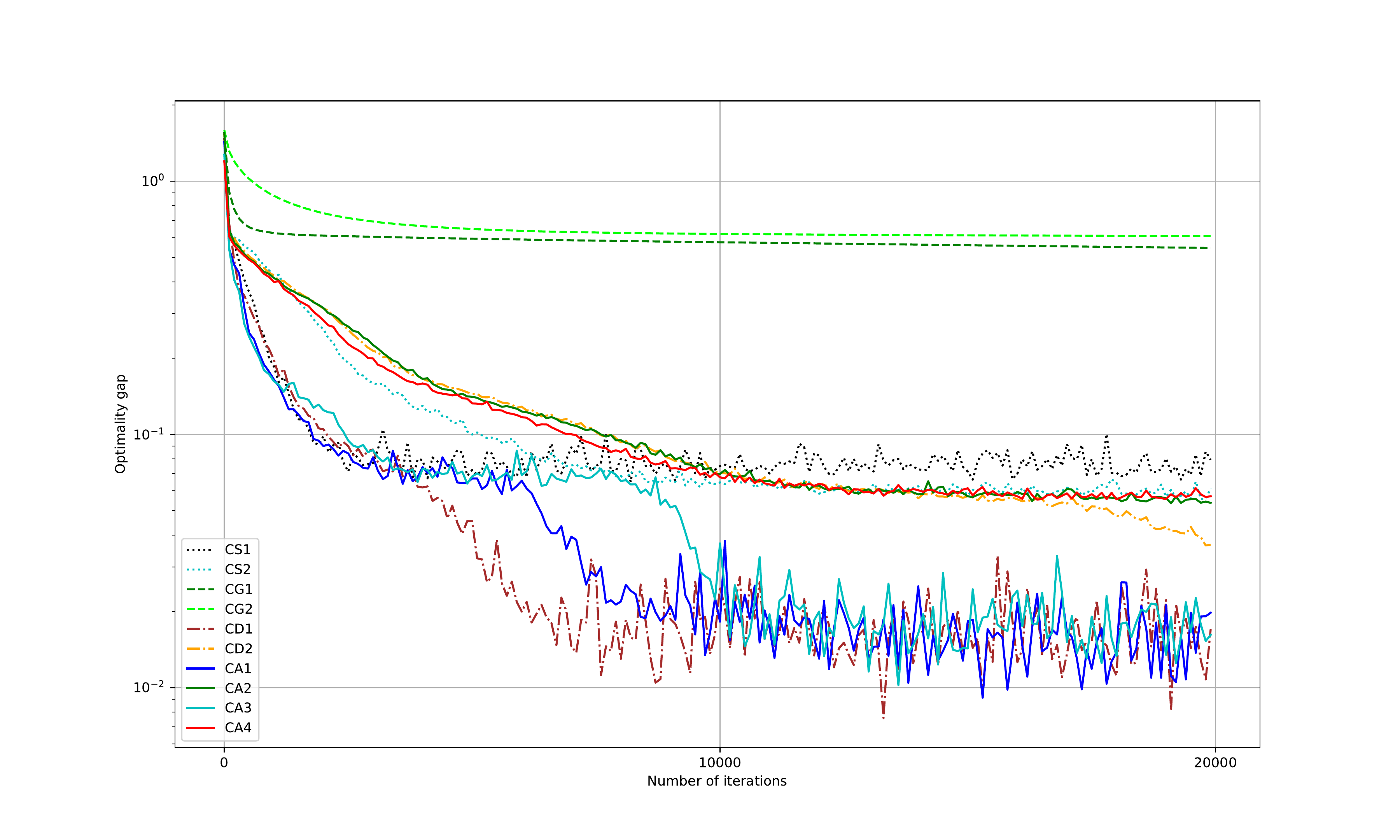}
 \DeclareGraphicsExtensions{.pdf}
\caption{Optimality gap versus number of iterations in the case of constant learning rates for the \texttt{digits} dataset \label{fig:digits-c1}}
\end{figure}

\begin{figure}[!t]
\centering
\includegraphics[width=3.0in]{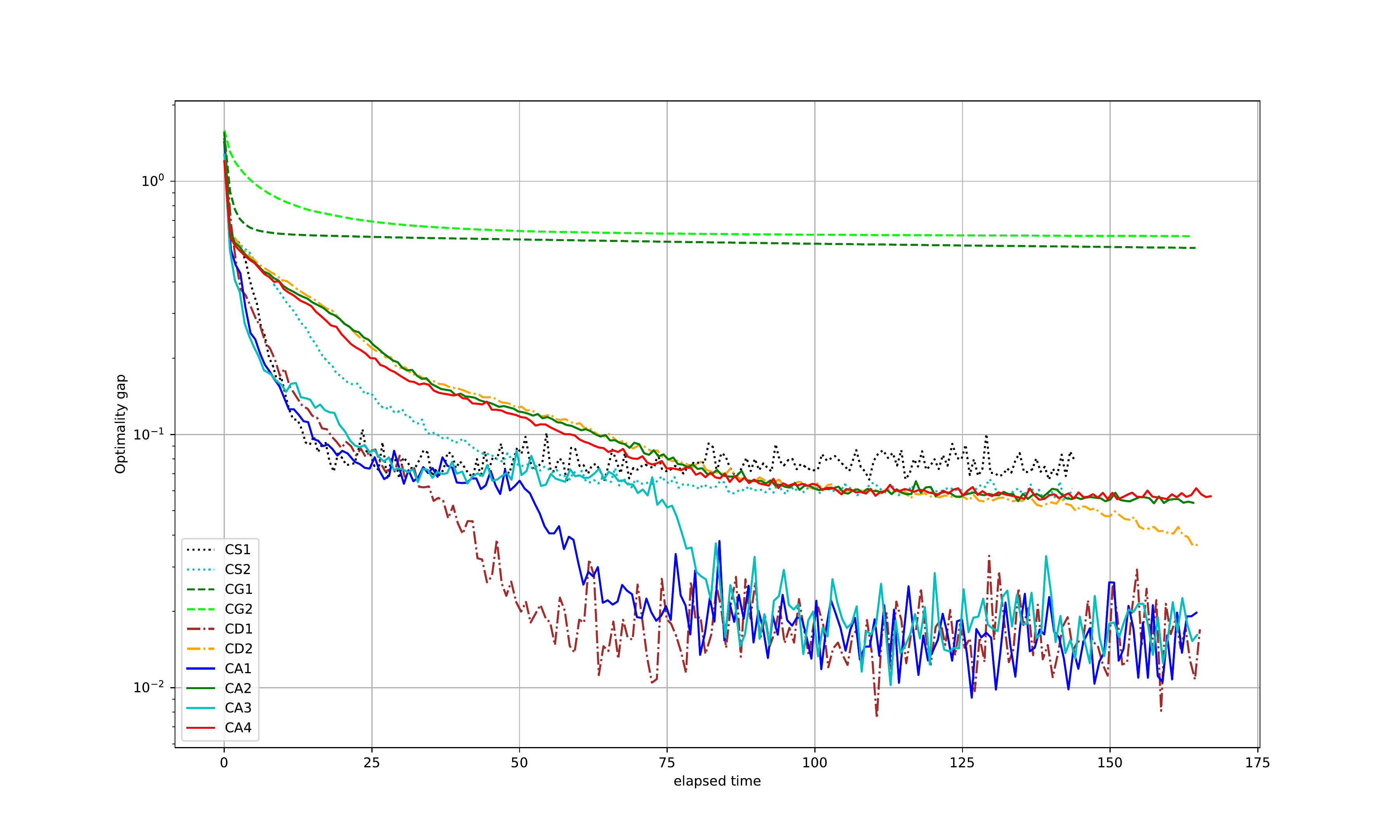}
 \DeclareGraphicsExtensions{.pdf}
\caption{Optimality gap versus elapsed time in the case of constant learning rates for the \texttt{digitis} dataset \label{fig:digits-c2}}
\end{figure}

Moreover, we examined the supervised learning performance. The classification model in this case was the Linear Support Vector Machine\footnote{\url{https://scikit-learn.org/stable/modules/generated/sklearn.svm.SVC.html}} (Linear SVM) provided by the scikit-learn 0.23.2 package. TABLE \ref{tab:C} shows the 5-fold cross validation scores of the \texttt{MNIST} and \texttt{digits} dimensionally reduced by the last point generated by each algorithm with the constant learning rates. This table indicates that the algorithm which sufficiently minimizes the optimality gap  also has high classification accuracy. For the \texttt{MNIST} dataset, since (CG1) does not converge to the optimal solution, its classification accuracy is also bad. Similarly, for the \texttt{digits} dataset, since (CG1) and (CG2) do not minimize the optimality gap, their classification accuracies are also bad.

\begin{table}[htbp]
\centering
\caption{The cross validation scores of the Linear SVM in the case of constant learning rates \label{tab:C}}
\begin{tabular}{|l|ll|}
\hline
 & \texttt{MNIST} & \texttt{digits} \\ \hline
CS1 & 0.8109 & 0.8658 \\
CS2 & 0.8104 & 0.8720 \\
CG1 & 0.5992 & 0.7929 \\
CG2 & 0.8164 & 0.7022 \\
CD1 & 0.7973 & 0.8736 \\
CD2 & 0.8078 & 0.8764 \\
CA1 & 0.8168 & 0.8764 \\
CA2 & 0.8099 & 0.8664 \\
CA3 & 0.7931 & 0.8520 \\
CA4 & 0.8131 & 0.8698 \\ \hline
\end{tabular}
\end{table}

\subsubsection{Diminishing learning rate}
Next, we compared ten algorithms (DS1)--(DA4) with diminishing learning rates, same as those used in the experiments of the Poincar{\'e} embeddings. Figs. \ref{fig:mnist-d1}--\ref{fig:mnist-d2} show the numerical results on the \texttt{MNIST} dataset, while Figs. \ref{fig:digits-d1}--\ref{fig:digits-d2} show the numerical results on the \texttt{digits} dataset. Figs. \ref{fig:mnist-d1}--\ref{fig:mnist-d2} indicate that Algorithm \ref{RAMSGrad} outperforms RSGD and RAdaGrad in every setting for the \texttt{MNIST} dataset. In particular, the behavior of (DD1) is the best of all and (DA1) performs comparably to (DD1). Moreover, Figs. \ref{fig:digits-d1}--\ref{fig:digits-d2} also show that Algorithm \ref{RAMSGrad} outperforms RSGD and RAdaGrad in every setting for the \texttt{digits} dataset. In particular, the behaviors of (DA2) and (DA4) are the best of all algorithms and the behavior of (DD2) is comparable to them. Meanwhile, Figs. \ref{fig:mnist-d1}--\ref{fig:digits-d2} show that RSGD and RAdaGrad often failed to reduce the optimal gap. Figs. 9 and 13 indicate that Algorithm \ref{RAMSGrad} and RAdam with a constant learning rate are superior to the other algorithms for the \texttt{MNIST} dataset. Meanwhile, Figs. 11 and 15 indicate that Algorithm \ref{RAMSGrad} and RAdam with a diminishing learning rate are superior to the other algorithms for the \texttt{digits} dataset.

\begin{figure}[!t]
\centering
\includegraphics[width=3.0in]{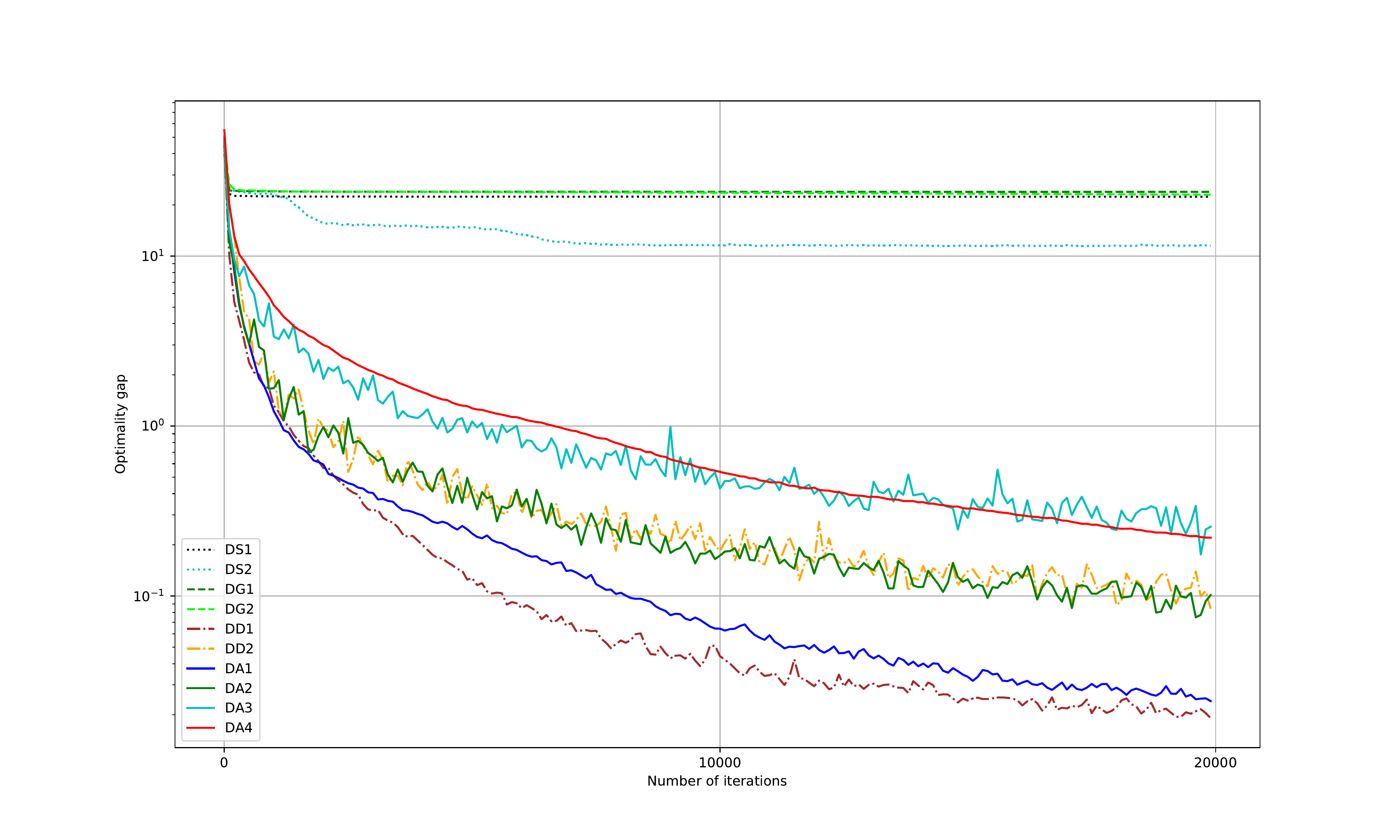}
 \DeclareGraphicsExtensions{.pdf}
\caption{Optimality gap versus number of iterations in the case of diminishing learning rates for the \texttt{MNIST} dataset \label{fig:mnist-d1}}
\end{figure}

\begin{figure}[!t]
\centering
\includegraphics[width=3.0in]{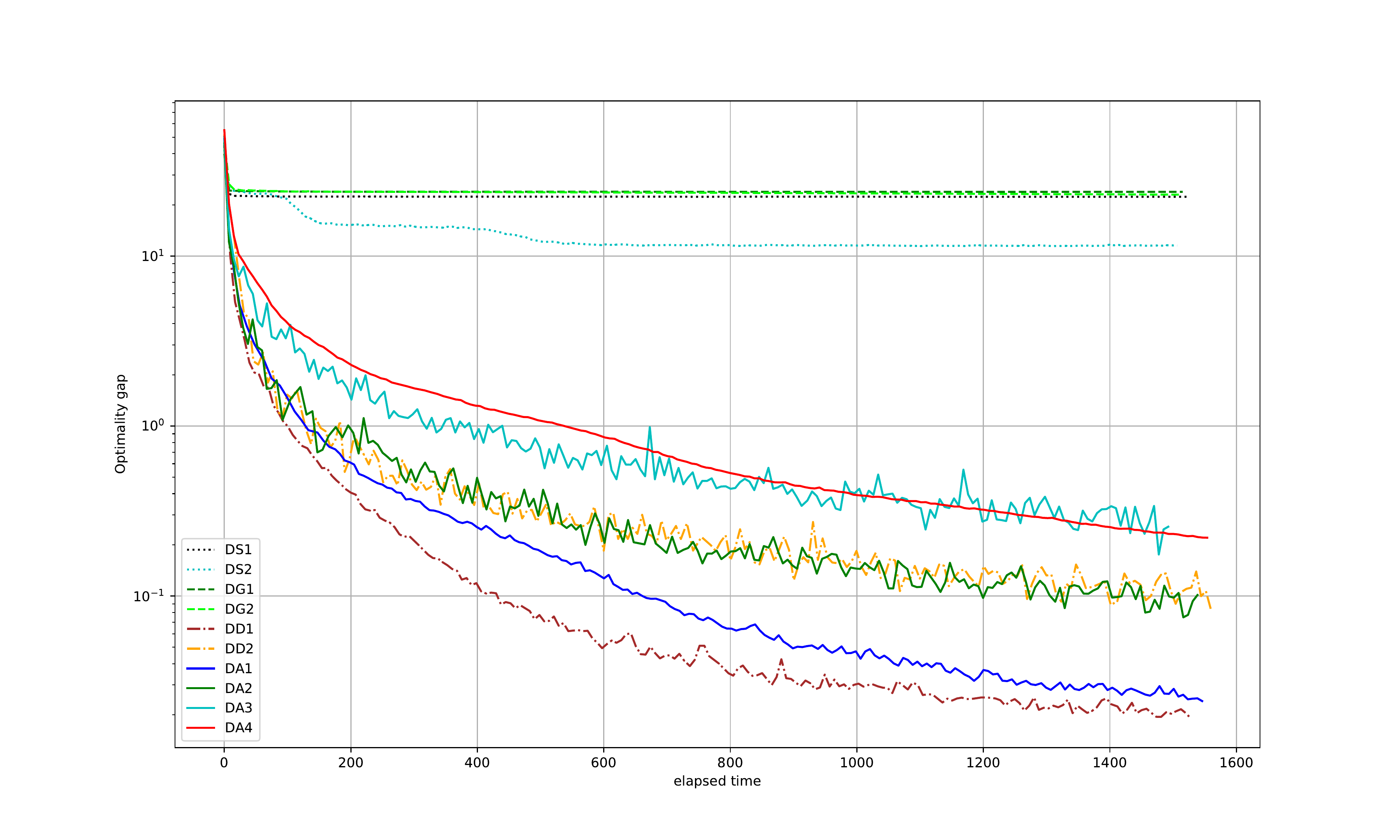}
 \DeclareGraphicsExtensions{.pdf}
\caption{Optimality gap versus elapsed time in the case of diminishing learning rates for the \texttt{MNIST} dataset \label{fig:mnist-d2}}
\end{figure}

\begin{figure}[!t]
\centering
\includegraphics[width=3.0in]{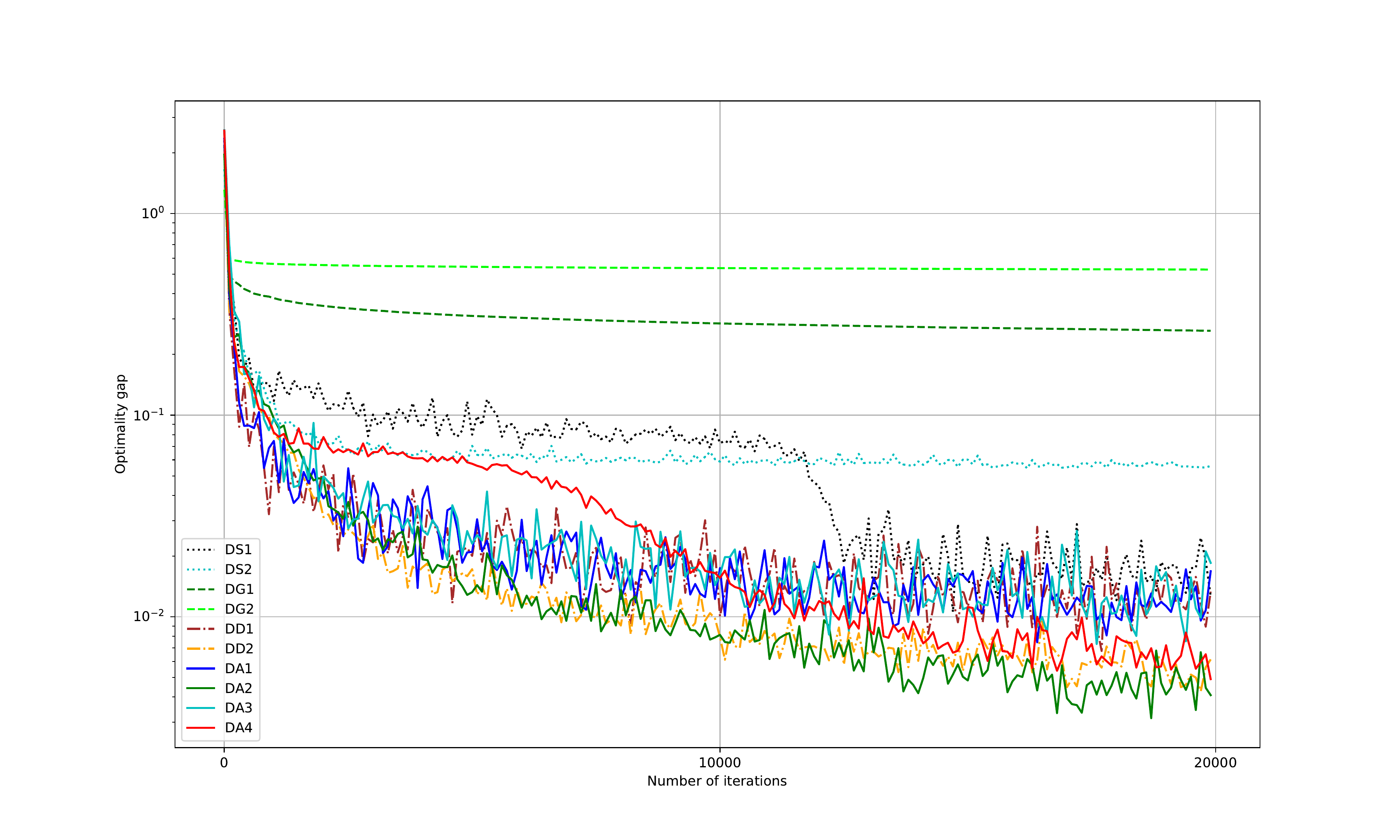}
 \DeclareGraphicsExtensions{.pdf}
\caption{Optimality gap versus number of iterations in the case of diminishing learning rates for the \texttt{digits} dataset \label{fig:digits-d1}}
\end{figure}

\begin{figure}[!t]
\centering
\includegraphics[width=3.0in]{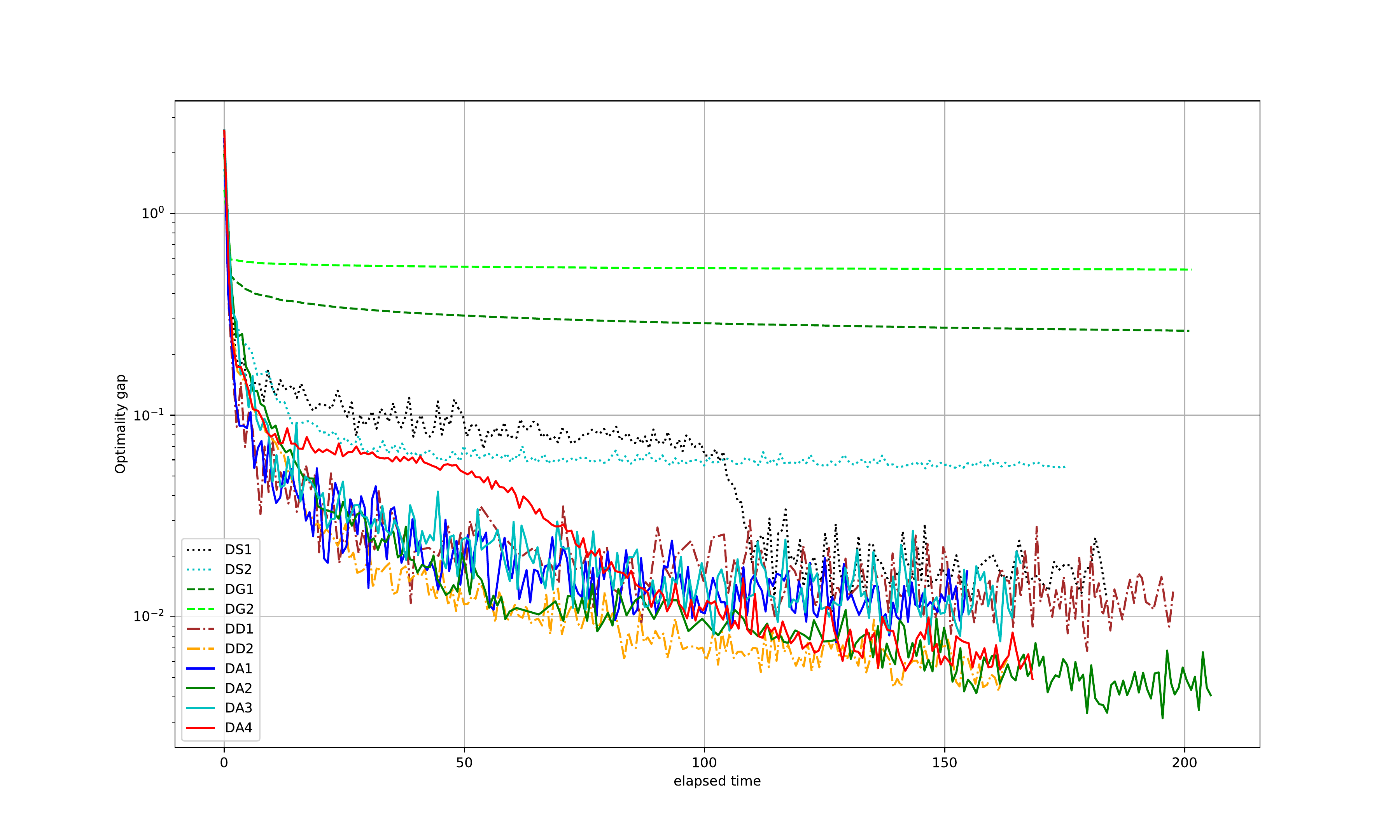}
 \DeclareGraphicsExtensions{.pdf}
\caption{Optimality gap versus elapsed time in the case of diminishing learning rates for the \texttt{digits} dataset \label{fig:digits-d2}}
\end{figure}

As with the constant learning rate, we examined the supervised learning performance of the Linear SVM. TABLE \ref{tab:D} shows the 5-fold cross validation scores of the \texttt{MNIST} and \texttt{digits} dimensionally reduced by the last point generated by each algorithm with the diminishing learning rates. This table indicates that the algorithm which sufficiently minimizes the optimality gap also has high classification accuracy. Since, (DS1), (DS2), (DG1), and (DG2) for the \texttt{MNIST} dataset, and (DG1) and (DG2) for the \texttt{digits} dataset do not minimize the optimality gap, their classification accuracies are also bad.

\begin{table}[htbp]
\centering
\caption{The cross validation scores of the Linear SVM in the case of diminishing learning rates \label{tab:D}}
\begin{tabular}{|l|ll|}
\hline
 & \texttt{MNIST} & \texttt{digits} \\ \hline
DS1 & 0.5719 & 0.8709 \\
DS2 & 0.6332 & 0.8714 \\
DG1 & 0.6546 & 0.7145 \\
DG2 & 0.5904 & 0.6856 \\
DD1 & 0.7955 & 0.8670 \\
DD2 & 0.7928 & 0.8692 \\
DA1 & 0.8133 & 0.8759 \\
DA2 & 0.7922 & 0.8764 \\
DA3 & 0.8239 & 0.8842 \\
DA4 & 0.8061 & 0.8742 \\ \hline
\end{tabular}
\end{table}

\section{Conclusion}
~\label{sec:CaFW}
This paper proposed modified RAMSGrad, a Riemannian adaptive optimization method, and presented its convergence analysis. The proposed algorithm solves the Riemannian optimization problem directly, and it can use both constant and diminishing learning rates. We applied it to Poincar{\'e} embeddings and a PCA problem. The numerical experiments showed that it converges to the optimal solution faster than RSGD and RAdaGrad, and it minimizes the objective function regardless of the initial learning rate. In particular, an experiment showed that the proposed algorithm with a constant learning rate is a good way of embedding the WordNet mammals subtree into a Poincar{\'e} subtree. Moreover, we showed that, in the PCA problem, the choice between using a constant or a diminishing learning rate depends on the dataset.

\section{Acknowledgment}
We are sincerely grateful to the editor and the anonymous referees for helping us improve the original manuscript.

\appendix
\section{Lemmas}
Zhang and Sra developed the following lemma in \cite[Lemma 5]{zhang2016first}.

\begin{lemma}
[Cosine inequality in Alexandrov spaces]~\label{lem:CiAs} Let $a, b, c$ be the sides (i.e., side lengths) of a geodesic triangle in an Alexandrov space whose curvature is bounded by $\kappa < 0$ and $A$ be the angle between sides $b$ and $c$. Then,
\begin{align*}
a^2 \leq \zeta{(\kappa ,c)}b^2 + c^2 - 2bc\cos{(A)},
\end{align*}
where
\begin{align*}
\quad \zeta{(\kappa ,c)} = \frac{\sqrt{|\kappa|c}}{\tanh{(\sqrt{|\kappa|c})}}.
\end{align*}
\end{lemma}

We will prove the following lemma. All relations between random variables hold almost surely.
\begin{lemma}~\label{lem:boundlem}
Suppose that Assumption \ref{asm:main} (A2) holds. We define $G:=\max_{t \in \mathcal{T},x \in X}\norm{\grad f_t(x)}{x}$. Let $(x_n)_{n \in \nat}$ and $(\hat{v}_n)_{n \in \nat}$ be the sequences generated by Algorithm 1. Then, for all $i \in \{1, 2, \cdots ,N\}$, and $k \in \nat$,
\begin{align}~\label{eq:mbound}
\norm{m_k^i}{x_k^i}^2 \leq G^2,
\end{align}
and
\begin{align}~\label{eq:vbound}
\sqrt{\hat{v}_k^i} \leq G.
\end{align}
\end{lemma}
\begin{proof}
First, we consider \eqref{eq:mbound}. The proof is by induction. For $k=1$, from the convexity of $\norm{\cdot}{x_1^i}^2$, we have
\begin{align*}
\norm{m_1^i}{x_1^i}^2 &\leq \norm{\beta_{11}\varphi_{x_{0}^i \rightarrow x_{1}^i}^i(m_0^i)+(1-\beta_{11})g_{t_1}^i}{x_1^i}^2 \\
&\leq \beta_{11}\norm{\varphi_{x_{0}^i \rightarrow x_{1}^i}^i(m_0^i)}{x_1^i}^2 + (1-\beta_{11})\norm{g_{t_1}^i}{x_1^i}^2 \\
&=(1-\beta_{11})\norm{g_{t_1}^i}{x_1^i}^2 \\
&\leq \norm{g_{t_1}^i}{x_1^i}^2 \\
&\leq G^2,
\end{align*}
where we have used $0 \leq \beta_{11} < 1$ and $\norm{g_{t_1}^i}{x_1^i} \leq G$. Suppose that $\norm{m_{k-1}^i}{x_{k-1}^i}^2 \leq G^2$. The convexity of $\norm{\cdot}{x_k^i}^2$, together with the definition of $m_k^i$, and $\norm{g_{t_k}^i}{x_k^i} \leq G$, guarantees that,
\begin{align*}
\norm{m_k^i}{x_k^i}^2 &\leq \beta_{1k}\norm{\varphi_{x_{k-1}^i \rightarrow x_{k}^i}^i(m_{k-1}^i)}{x_k^i}^2 + (1-\beta_{1k})\norm{g_{t_k}^i}{x_k^i}^2 \\
&\leq \beta_{1k}\norm{m_{k-1}^i}{x_{k-1}^i}^2 + (1-\beta_{1k})G^2 \\
&\leq \beta_{1k}G^2 + (1-\beta_{1k})G^2 \\
&= G^2.
\end{align*}
Thus, induction ensures that, for all $k \in \nat$,
\begin{align*}
\norm{m_k^i}{x_k^i}^2 \leq G^2.
\end{align*}
\eqref{eq:vbound} can be proven in same way as \eqref{eq:mbound}.
\end{proof}

\section{Proof of Theorem \ref{thm:main}}
~\label{sec:apendpt}
\begin{proof}
[Proof of Theorem \ref{thm:main}] Note that
\begin{align*}
y_{k+1}^i:=\exp_{x_k^i}^i\left(-\alpha_k\dfrac{m_k^i}{\sqrt{\hat{v}_k^i}}\right).
\end{align*}
Thus, we will consider a geodesic triangle consisting of three points $x_k^i$, $x_\ast^i$, and $y_{k+1}^i$. Let the length of each side be $a$, $b$, and $c$, respectively, such that
\begin{align}~\label{eq:abc}
\begin{cases}
a := d^i(y_{k+1}^i,x_\ast^i) \\
b := d^i(y_{k+1}^i,x_k^i) \\
c := d^i(x_k^i,x_\ast^i)
\end{cases}.
\end{align}
It follows that
\begin{align*}
\cos{(\angle y_{k+1}^ix_k^ix_\ast^i)} &:= \frac{\ip{\log_{x_k^i}^i(y_{k+1}^i)}{\log_{x_k^i}^i(x_\ast^i)}{x_k^i}}{\norm{\log_{x_k^i}^i(y_{k+1}^i)}{x_k^i}\norm{\log_{x_k^i}^i(x_\ast^i)}{x_k^i}} \\
&= \frac{\ip{-\alpha_k\dfrac{m_k^i}{\sqrt{\hat{v}_k^i}}}{\log_{x_k^i}^i(x_\ast^i)}{x_k^i}}{d^i(y_{k+1}^i,x_k^i)d^i(x_k^i,x_\ast^i)}.
\end{align*}
Using Lemma \ref{lem:CiAs} with \eqref{eq:abc} and the definition of $\Pi_{X_i}$, we have
\begin{align*}
&d^i(x_{k+1}^i,x_\ast^i)^2 \\
&\leq d^i(y_{k+1}^i,x_\ast^i)^2 \\
&\leq \zeta(\kappa^i,d^i(x_k^i,x_\ast^i))d^i(y_{k+1}^i,x_k^i)^2 + d^i(x_k^i,x_\ast^i)^2 \\
&\quad -2d^i(y_{k+1}^i,x_k^i)d^i(x_k^i,x_\ast^i)\frac{\ip{-\alpha_k\dfrac{m_k^i}{\sqrt{\hat{v}_k^i}}}{\log_{x_k^i}^i(x_\ast^i)}{x_k^i}}{d^i(y_{k+1}^i,x_k^i)d^i(x_k^i,x_\ast^i)},
\end{align*}
which, together with the definition of $y_{k+1}^i$, implies that
\begin{align*}
\ip{-m_k^i}{\log_{x_k^i}^i(x_\ast^i)}{x_k^i} &\leq \dfrac{\sqrt{\hat{v}_k^i}}{2\alpha_k}(d^i(x_k^i,x_\ast^i)^2 - d^i(x_{k+1}^i,x_\ast^i)^2) \\
&\quad + \zeta(\kappa^i,d^i(x_k^i,x_\ast^i))\dfrac{\alpha_k}{2\sqrt{\hat{v}_k^i}}\norm{m_k^i}{x_k^i}^2.
\end{align*}
Plugging $m_k^i=\beta_{1k}\varphi_{x_{k-1}^i \rightarrow x_{k}^i}^i(m_{k-1}^i)+(1-\beta_{1k})g_{t_k}^i$ into the above inequality and using (A1), we obtain
\begin{align}
\begin{split}~\label{eq:convexineq}
&\ip{-g_{t_k}^i}{\log_{x_k^i}(x_\ast^i)}{x_k^i} \\
&\leq \dfrac{\sqrt{\hat{v}_k^i}}{2\alpha_k(1-\beta_{1k})}\left(d^i(x_k^i,x_\ast^i)^2 - d^i(x_{k+1}^i,x_\ast^i)^2\right) \\
&\quad + \dfrac{\zeta(\kappa^i,D)}{2(1-\beta_{1k})}\dfrac{\alpha_k}{\sqrt{\hat{v}_k^i}}\norm{m_k^i}{x_k^i}^2 \\
&\quad + \frac{\beta_{1k}}{1-\beta_{1k}}\ip{\varphi_{x_{k-1}^i \rightarrow x_{k}^i}^i(m_{k-1}^i)}{\log_{x_k^i}(x_\ast^i)}{x_k^i}.
\end{split}
\end{align}
Since (A2) implies that $f$ is geodesically convex with $g(x)=\left(g^i(x^i)\right):=\grad{f(x)}$, we have
\begin{align*}
f(x_k) - f(x_\ast) &\leq \ip{-g(x_k)}{\log_{x_k}(x_\ast)}{x_k} \\
&= \sum_{i=1}^N\ip{-g^i(x_k^i)}{\log_{x_k^i}^i(x_\ast^i)}{x_k^i}.
\end{align*}
Summing the above equality from $k = 1$ to $n$, we obtain
\begin{align}~\label{eq:temp1}
\frac{1}{n}\sum_{k=1}^nf(x_k) - f(x_\ast) \leq \frac{1}{n}\sum_{k=1}^n\sum_{i=1}^N\ip{-g^i(x_k^i)}{\log_{x_k^i}^i(x_\ast^i)}{x_k^i}.
\end{align}
Furthermore, the linearity of the Riemannian gradient ensures that
\begin{align*}
&\expect{}{\ip{-g_{t_k}^i}{\log_{x_k^i}^i(x_\ast^i)}{x_k^i}} \\
&= \expect{}{\mathbb{E}\left[ \ip{-g_{t_k}^i}{\log_{x_k^i}^i(x_\ast^i)}{x_k^i} \relmiddle| t_{[k-1]} \right]} \\
&= \expect{}{\ip{-\mathbb{E}\left[ g_{t_k}^i\relmiddle| t_{[k-1]} \right]}{\log_{x_k^i}^i(x_\ast^i)}{x_k^i}} \\
&= \expect{}{\ip{-g^i(x_k^i)}{\log_{x_k^i}^i(x_\ast^i)}{x_k^i}},
\end{align*}
which, together with \eqref{eq:temp1}, implies that
\begin{align*}
& \expect{}{\frac{1}{n}\sum_{k=1}^nf(x_k) - f(x_\ast)} \\
&\leq \frac{1}{n}\expect{}{\sum_{k=1}^n\sum_{i=1}^N\ip{-g^i(x_k^i)}{\log_{x_k^i}^i(x_\ast^i)}{x_k^i}} \nonumber \\
&= \frac{1}{n}\expect{}{\sum_{k=1}^n\sum_{i=1}^N\ip{-g_{t_k}^i}{\log_{x_k^i}^i(x_\ast^i)}{x_k^i}}.
\end{align*}
From \eqref{eq:convexineq} and the above inequality, we have
\begin{align}
\begin{split}~\label{eq:exdif}
&\expect{}{\frac{1}{n}\sum_{k=1}^nf(x_k) - f(x_\ast)} \\
&\leq \frac{1}{n}\expect{}{\sum_{k=1}^n\sum_{i=1}^N\dfrac{\sqrt{\hat{v}_k^i}}{2\alpha_k(1-\beta_{1k})}\left(d^i(x_k^i,x_\ast^i)^2 - d^i(x_{k+1}^i,x_\ast^i)^2\right)} \\
&\quad + \frac{1}{n}\expect{}{\sum_{k=1}^n\sum_{i=1}^N\dfrac{\zeta(\kappa^i,D)}{2(1-\beta_{1k})}\dfrac{\alpha_k}{\sqrt{\hat{v}_k^i}}\norm{m_k^i}{x_k^i}^2} \\
&\quad + \frac{1}{n}\expect{}{\sum_{k=1}^n\sum_{i=1}^N\frac{\beta_{1k}}{1-\beta_{1k}}\ip{\varphi_{x_{k-1}^i \rightarrow x_{k}^i}^i(m_{k-1}^i)}{\log_{x_k^i}(x_\ast^i)}{x_k^i}}.
\end{split}
\end{align}

Here, let us consider the first term of the left-hand side of \eqref{eq:exdif}. We note that from the assumption for all $k \in \nat$, $\alpha_{k}(1-\beta_{1k}) \leq \alpha_{k-1}(1-\beta_{1,k-1})$, and $\beta_{1k} \leq \beta_{1,k-1}$,
\begin{align*}
\alpha_{k}(1-\beta_{1k}) \leq \alpha_{k-1}(1-\beta_{1,k-1}) \leq \alpha_{k-1}(1-\beta_{1k}),
\end{align*}
which implies $\alpha_k \leq \alpha_{k-1}$. Using $\beta_{1k} \leq \beta_{11}$, $\alpha_k \leq \alpha_{k-1}$, $\sqrt{\hat{v}_{k}^i} \geq \sqrt{\hat{v}_{k-1}^i}$, and $\alpha_{k}(1-\beta_{1k}) \leq \alpha_{k-1}(1-\beta_{1,k-1})$ for all $k \in \nat$, together with (A1), we have that
\begin{align*}
&\sum_{k=1}^n\sum_{i=1}^N\dfrac{\sqrt{\hat{v}_k^i}}{2\alpha_k(1-\beta_{1k})}(d^i(x_k^i,x_\ast^i)^2 - d^i(x_{k+1}^i,x_\ast^i)^2) \\
&\leq \frac{1}{2(1-\beta_{11})}\sum_{i=1}^N\left[\sum_{k=2}^n\left(\dfrac{\sqrt{\hat{v}_k^i}}{\alpha_k}-\dfrac{\sqrt{\hat{v}_{k-1}^i}}{\alpha_{k-1}}\right)d^i(x_k^i,x_\ast^i)^2 \right. \\
&\left. \quad +\dfrac{\sqrt{\hat{v}_1^i}}{\alpha_1}d^i(x_1^i,x_\ast^i)^2\right] \\
&\leq \frac{1}{2(1-\beta_{11})}\sum_{i=1}^N\left[\sum_{k=2}^n\left(\dfrac{\sqrt{\hat{v}_k^i}}{\alpha_k}-\dfrac{\sqrt{\hat{v}_{k-1}^i}}{\alpha_{k-1}}\right)D^2+\dfrac{\sqrt{\hat{v}_1^i}}{\alpha_1}D^2\right] \\
&= \frac{D^2}{2(1-\beta_{11})}\sum_{i=1}^N\dfrac{\sqrt{\hat{v}_n^i}}{\alpha_n} \\
&\leq \frac{NGD^2}{2\alpha_n(1-\beta_{11})},
\end{align*}
where the last inequality is guaranteed by Lemma \ref{lem:boundlem}. Namely,
\begin{align}~\label{eq:1t}
\begin{split}
&\expect{}{\sum_{k=1}^n\sum_{i=1}^N\dfrac{\sqrt{\hat{v}_k^i}}{2\alpha_k(1-\beta_{1k})}(d^i(x_k^i,x_\ast^i)^2 - d^i(x_{k+1}^i,x_\ast^i)^2)} \\
&\leq \dfrac{NGD^2}{2\alpha_n(1-\beta_{11})}.
\end{split}
\end{align}
Next, let us consider the second term of the left-hand side of \eqref{eq:exdif}. From $\sqrt{\epsilon} \leq \sqrt{\hat{v}_k^i}$ and Lemma \ref{lem:boundlem}, we have
\begin{align*}
&\sum_{k=1}^n\sum_{i=1}^N\dfrac{\zeta(\kappa^i,D)}{2(1-\beta_{1k})}\dfrac{\alpha_k}{\sqrt{\hat{v}_k^i}}\norm{m_k^i}{x_k^i}^2 \\
&\leq \dfrac{G^2}{2\sqrt{\epsilon}(1-\beta_{11})}\sum_{i=1}^N\zeta(\kappa_i,D)\sum_{k=1}^n\alpha_k.
\end{align*}
Namely,
\begin{align}~\label{eq:2t}
\begin{split}
&\expect{}{\sum_{k=1}^n\sum_{i=1}^N\dfrac{\zeta(\kappa^i,D)}{2(1-\beta_{1k})}\dfrac{\alpha_k}{\sqrt{\hat{v}_k^i}}\norm{m_k^i}{x_k^i}^2} \\
&\leq \dfrac{G^2}{2\sqrt{\epsilon}(1-\beta_{11})}\sum_{i=1}^N\zeta(\kappa_i,D)\sum_{k=1}^n\alpha_k.
\end{split}
\end{align}
Now, let us consider the third term of the left-hand side of \eqref{eq:exdif}. Applying the Cauchy-Schwarz inequality to the term and using (A1) and Lemma \ref{lem:boundlem}, it follows that
\begin{align*}
&\sum_{k=1}^n\sum_{i=1}^N\frac{\beta_{1k}}{1-\beta_{1k}}\ip{\varphi_{x_{k-1}^i \rightarrow x_{k}^i}^i(m_{k-1}^i)}{\log_{x_k^i}(x_\ast^i)}{x_k^i} \\
&\leq \sum_{k=1}^n\sum_{i=1}^N\frac{\beta_{1k}}{1-\beta_{1k}}\norm{\varphi_{x_{k-1}^i \rightarrow x_{k}^i}^i(m_{k-1}^i)}{x_k^i} \norm{\log_{x_k^i}(x_\ast^i)}{x_k^i} \\
&\leq \frac{NGD}{1-\beta_{11}}\sum_{k=1}^n\beta_{1k}.
\end{align*}
Namely,
\begin{align}~\label{eq:3t}
\begin{split}
&\expect{}{\sum_{k=1}^n\sum_{i=1}^N\frac{\beta_{1k}}{1-\beta_{1k}}\ip{\varphi_{x_{k-1}^i \rightarrow x_{k}^i}^i(m_{k-1}^i)}{\log_{x_k^i}(x_\ast^i)}{x_k^i}} \\
&\leq \frac{NGD}{1-\beta_{11}}\sum_{k=1}^n\beta_{1k}.
\end{split}
\end{align}
Finally, together with \eqref{eq:exdif}, \eqref{eq:1t}, \eqref{eq:2t}, and \eqref{eq:3t}, we have
\begin{align*}
&\expect{}{\frac{1}{n}\sum_{k=1}^nf(x_k) - f(x_\ast)} \\
&\leq \dfrac{NGD^2}{2(1-\beta_{11})}\frac{1}{n\alpha_n} + \dfrac{G^2}{2\sqrt{\epsilon}(1-\beta_{11})}\sum_{i=1}^N\zeta(\kappa_i,D)\frac{1}{n}\sum_{k=1}^n\alpha_k \\
&\quad +\frac{NGD}{1-\beta_{11}}\frac{1}{n}\sum_{k=1}^n\beta_{1k}.
\end{align*}
This complete the proof.
\end{proof}

\section{Proof of Corollary \ref{cor:CLR} and \ref{cor:DLR}}
~\label{sec:apendpc}
\begin{proof}
[Proof of Corollary \ref{cor:CLR}] The learning rates $\alpha_{n}:=\alpha$ and $\beta_{1n} := \beta$ satisfy for all $n \in \nat$, $\beta_{1n} \leq \beta_{1,n-1}$ and $\alpha_n(1-\beta_{1n}) \leq \alpha_{n-1}(1-\beta_{1,n-1})$. Let us define
\begin{align*}
C_1 := \dfrac{G^2}{\sqrt{\epsilon}(1-\beta_{11})}\sum_{i=1}^N\zeta(\kappa_i,D) > 0,
\end{align*}
and
\begin{align*}
C_2:= \frac{NGD}{1-\beta_{11}}.
\end{align*}
Using the definitions of $C_1$ and $C_2$, \eqref{eq:main} can be written as
\begin{align*}
\expect{}{\frac{1}{n}\sum_{k=1}^nf(x_k) - f(x_\ast)} \leq \dfrac{NGD^2}{2\alpha(1-\beta_{11})}\frac{1}{n} + C_1\alpha + C_2\beta.
\end{align*}
This complete the proof.
\end{proof}

\begin{proof}
[Proof of Corollary \ref{cor:DLR}] Let $\alpha_n = 1/n^\eta$ $(\eta \in [1/2,1))$ and $(\beta_{1n})_{n \in \nat}$ satisfies $\beta_{1n} \leq \beta_{1,n-1}$ and $\alpha_n(1-\beta_{1n}) \leq \alpha_{n-1}(1-\beta_{1,n-1})$ for all $n \in \nat$, and $\sum_{k=1}^\infty\beta_{1k}<\infty$. First, we obviously have
\begin{align}~\label{eq:temp2}
\lim_{n \to \infty}\frac{1}{n}\sum_{k=1}^n\beta_{1k} \leq \lim_{n \to \infty}\frac{B_1}{n} = 0,
\end{align}
where $B_1 := \sum_{k=1}^\infty\beta_{1k} < \infty$. We have that
\begin{align*}
\lim_{n \to \infty}\frac{1}{n\alpha_n}=\lim_{n \to \infty}\frac{1}{n^{1-\eta}}=0.
\end{align*}
Furthermore, we have
\begin{align}~\label{eq:order1}
\frac{1}{n}\sum_{k=1}^n\alpha_k &\leq \frac{1}{n}\left(1 + \int_1^n\frac{dt}{t^\eta}\right) \leq \frac{1}{1-\eta}\frac{1}{n^{1-\eta}}.
\end{align}
This, together with \eqref{eq:main}, \eqref{eq:temp2}, and \eqref{eq:order1}, proves the assertion of Corollary \ref{cor:DLR}.
\end{proof}

\bibliographystyle{abbrv}
\bibliography{biblib}

\end{document}